\documentclass[12pt,reqno,a4paper]{amsart}
\usepackage{amsmath,amssymb,amsthm,amsfonts}
\usepackage[dvips]{graphicx}
\usepackage{mathrsfs}

\usepackage{color}

\usepackage[
  a4paper,
  left=1.2in,
  right=1.2in,
  top=1.2in,
  bottom=1.2in,
  headheight=20pt,
  headsep=0.25in,
  heightrounded
]{geometry}

\numberwithin{equation}{section}

{\theoremstyle{definition}
\newtheorem{defn}{Definition}[section]}
\newtheorem{theorem}{Theorem}[section]
\newtheorem{proposition}[theorem]{Proposition}

\newtheorem{lemma}[theorem]{Lemma}

{\theoremstyle{definition}
{
\newtheorem{remark}[theorem]{Remark}

}}

\newcommand{\Vs}{{\mathscr{V}}}
\newcommand{\Bs}{{\mathscr{B}}}
\newcommand{\Us}{{\mathscr{U}}}

\newcommand{\A}{{\mathcal A}}

\newcommand{\CC}{{\mathcal C}}

\newcommand{\GG}{{\mathcal G}}

\newcommand{\LL}{{\mathcal L}}

\newcommand{\UU}{{\mathcal U}}

\newcommand{\Cc}{{\mathbb{C}}}

\newcommand{\Ee}{{\mathbb{E}}}

\newcommand{\Ii}{{\mathbb{I}}}

\newcommand{\Kk}{{\mathbb{K}}}

\newcommand{\Nn}{{\mathbb{N}}}

\newcommand{\Rr}{{\mathbb{R}}}

\newcommand{\Tt}{{\mathbb{T}}}

\newcommand{\Zz}{{\mathbb{Z}}}

\def\e{\mathrm{e}}

\def\diag{\operatorname{diag}}

\def\GL{\operatorname{GL}}

\def\SL{\operatorname{SL}}

\def\SLZ{\SL_d(\Zz)}
\def\SLR{\SL_d(\Rr)}

\newcommand{\abs}[1]{\left| #1\right|}

\newcommand{\id}  {\operatorname{id}}

\newcommand{\rot}{\operatorname{rot}}

\newcommand{\comment}[1]{}
\def\veck{{\text{\boldmath$k$}}}

\newcommand{\vertiii}[1]{{\left\vert\kern-0.25ex\left\vert\kern-0.25ex\left\vert #1 
    \right\vert\kern-0.25ex\right\vert\kern-0.25ex\right\vert}}

\def\veck{{\text{$k$}}}

\def\vecalf{{\text{$\alpha$}}}

\def\vecomega{{\text{$\omega$}}}

\begin{document}

\title[Brjuno condition and best approximations]{Brjuno condition through best approximations and the linearization problem}
\date{\today}

\author[Chevallier]{Nicolas Chevallier}
\address{Département de Mathématiques - IRIMAS,
Université de Haute-Alsace,
18 Rue des Frères Lumière, 68200 Mulhouse, France}
\email{nicolas.chevallier@uha.fr}

\author[Lopes Dias]{Jo\~{a}o Lopes Dias}
\address{
ISEG Research, ISEG Lisbon School of Economics and Management, Universidade de Lisboa, Lisbon, Portugal
}
\email{jldias@iseg.ulisboa.pt}

\author[Gaiv\~ao]{Jos\'e Pedro Gaiv\~ao}
\address{Departamento de Matem\'atica, ISEG\\
Universidade de Lisboa\\
Rua do Quelhas 6, 1200-781 Lisboa, Portugal}
\email{jpgaivao@iseg.ulisboa.pt}

\author[Marnat]{Antoine Marnat}
\address{Laboratoire d'Analyse et de Mathématiques Appliquées\\
Université Paris Est Créteil - Val de Marne\\
61, avenue du Général de Gaulle\\
Bureau P4 - 407\\
94 010 CRETEIL Cedex, France}
\email{antoine.marnat@u-pec.fr}

\author[Moshchevitin]{Nikolay Moshchevitin}
\address{Institut für Diskrete Mathematik und Geometrie\\
Technische Universität Wien\\
Freihaus, Wiedner Hauptstraße 8, A-1040, Wien, Österreich}
\email{nikolai.moshchevitin@tuwien.ac.at, moshchevitin@gmail.com}

\begin{abstract}

We consider the classical analytic linearization problem for vector fields on the torus $\Tt^d$ close to a constant vector field $\omega$. Our goals are twofold. First, we provide a geometric framework in which the arithmetic condition governing analytic linearization arises naturally from the orbit of a unimodular lattice associated with $\omega$ under a diagonal flow on $\SL(d,\Zz)\backslash \SL(d,\Rr)$. Within this framework, a summability condition emerges
as the natural criterion for convergence. We prove that it is equivalent to several classical formulations of the Brjuno condition for linear forms, including those involving best approximation vectors and switching times of the diagonal flow.
As a byproduct, we obtain a new quantitative linearization theorem with fully explicit estimates. In particular, the loss of analyticity of the conjugacy is controlled by a Brjuno function.

\end{abstract}
\maketitle
\tableofcontents

\section{Introduction}\label{section:Introduction}

The linearization problem is a central theme in quasi-periodic dynamics since it appears in several  contexts, including the existence of Siegel disks, the conjugacy of toral diffeomorphisms to rigid translations, the reducibility of linear Schr\"odinger cocycles and the persistence of invariant tori in nearly integrable Hamiltonian systems. 

A common feature underlying these problems is the reduction to a cohomological equation whose solvability is obstructed by small divisors governed by the arithmetic properties of the frequency vector $\omega$; cf. the pioneering works of Siegel, Kolmogorov, Arnold and Moser ~\cite{Siegel1942,Kolmogorov1954,Arnold1963,Moser1962}.
It is well known that the solvability of these problems  depend crucially on the arithmetic nature of~$\omega$: whereas Diophantine frequencies typically ensure solvability of the associated cohomological equation and lead to analytic linearization through KAM-type and renormalization schemes,
Liouvillean frequencies give rise to severe small divisor obstructions and, in general, prevent such a conjugacy.

%
%
In dimension one, Brjuno~\cite{Brjuno,Brjuno2} identified a number theoretical sufficient condition for analytic linearization (see ~\eqref{eq: BC 1-dim} below), while Yoccoz~\cite{Yoccoz4,Yoccoz2,Yoccoz1995}  further proved its optimality in specific contexts. In higher dimensions, several extensions have been obtained; however, their intrinsic  meaning remains less transparent.

The first purpose of the present work is to provide a geometric framework in which the relevant arithmetic condition arises naturally from the structure of resonances.
Our approach consists of a sequential elimination of Fourier modes by means of near-identity coordinate changes. 
At each step,  say the $n$-th step, one removes  far-from-resonant modes, while retaining a cone of near-resonant modes determined by a parameter $\sigma_n>0$:
\[
K_{\sigma_n}^+(\omega)=\{k\in\Zz^d : |k\cdot\omega|\le \sigma_n |k|\},
\]
where $|\cdot|$ stands for the $\ell_1$-norm. 
The widths $\sigma_n$ of these cones go to zero and control the loss of analyticity in the elimination process.

Writing $\omega=(\alpha,1)\in\Rr^{d-1}\times \Rr$, the key observation is that a good choice of these cones is dictated by the geometry of the orbit of the lattice $\Lambda_\omega=\Zz^d \left[\begin{smallmatrix}I&\alpha^\top\\0&1\end{smallmatrix}\right]$
under the diagonal flow $\Phi^t$ on $\SL(d,\Zz)\backslash\SL(d,\Rr)$. 
More precisely, the size of the shortest vector function
$\delta(\Phi^t(\Lambda_\omega))$ of the lattice $\Phi^t(\Lambda_\omega)$
encodes the arithmetical properties of $\omega$, and the function
\[
\Delta(t)=t+\log\delta(\Phi^t(\Lambda_\omega))
\]
measures the effective width $\sigma(t)$ of resonances at scale $t$.
The function $\Delta$ is piecewise affine, with slopes $0$ and $d$, and its switching times $(\tau_n)_{n\in\Nn}$ determine best approximation vectors of the linear form associated with $\omega$ (see section~\ref{sec:diophantine approx}). 
These times determine the sequence of cone widths in the elimination scheme.  
We observe that in the case $\omega$ is a Diophantine vector (i.e.  $|k\cdot\omega|>C/|k|^{d-1+\beta}$,  $k\in\Zz^{d}\setminus\{0\}$, for some $C>0$ and $\beta\geq0$), the function $\Delta(t)$ grows at least linearly in $t$,  i.e.
$$
\Delta(t)>c+\frac{d}{d+\beta} t
$$ 
for some $c\in\Rr$ (see \cite[Proposition 2.3]{jld10}).
We are interested  in describing conditions on vectors $\omega$ that induce slower growths, corresponding to faster approaches to zero of $\delta(\Phi^{\tau_n}(\Lambda_\omega))$.

In dimension $d=2$, the switching times coincide with those of the classical continued fraction algorithm of $\alpha$ (where $(q_n)_{n\geq0}$ are the denominators of the rational convergents), and the condition reduces to the usual one-dimensional Brjuno condition
\begin{equation}\label{eq: BC 1-dim}
\sum_{n\geq0}\frac{\log q_{n+1}}{q_n} <\infty.
\end{equation}
This is a particular case of the original Brjuno multidimensional condition: we say that $\omega$ is Brjuno if 
$$
O_\omega:=\sum_{j\geq0}\frac1{2^j}
\abs{\log\min_{0<|k|<2^j}|k\cdot\omega|} <\infty,
$$
where $k$ is taken over $\Zz^d$.

In higher dimensions, the behaviour of the diagonal flow replaces the classical continued fraction expansion and provides a multidimensional analogue.  
To explain this precisely,  we now introduce two quantities related to the function $\Delta$:
$$
I_\omega:=\int_0^\infty e^{-\Delta(t)}\,dt,
\qquad
B_\omega:=\sum_{n\geq0}e^{-\Delta(\tau_n)}\tau_{n+1}.
$$
The first quantity has a natural dynamical interpretation in terms of the diagonal flow, while the second expresses the same properties through the sequence of switching times.  

Let $(p_n)_{n\geq0}$ be the sequence of best approximation vectors to the linear form $\omega$ (see section~\ref{sec:diophantine approx}).
We thus define yet another quantity related to the arithmetical properties of $\omega$, namely
\[
 L_{\omega}:=\sum_{n\geq 0}\frac{|\log |p_n\cdot\omega||}{|\hat p_n|}.
\]


In  Theorem~ \ref{prop:Br2}  we prove the relation between these different quantities, and discuss their relation with
best simultaneous Diophantine approximations.
As a consequence, we obtain the following characterization of Brjuno vectors.

\begin{theorem}
$\omega$ is Brjuno iff 
$I_\omega<\infty$ iff $B_\omega<\infty$ iff $L_\omega<\infty$.
\end{theorem}

Moreover, we show that the set of vectors satisfying a condition obtained by
Golse and Lochak~\cite{GoLo}, in a previous attempt to characterize the Brjuno condition for $d\geq3$ through Diophantine approximations, is strictly larger than the set of Brjuno vectors (Section~\ref{sec:non BC}). In fact, the Hausdorff dimension of the difference is exactly $d-2$ (Theorem~\ref{prop: Hausdorff dim}).

\bigskip

A second objective of this work is to establish a quantitative linearization result in a specific setting, building on the geometric framework developed above. 
The problem of analytic linearization of flows on the torus $\Tt^d := \Rr^d / 2\pi \Zz^d$, with $d \geq 2$, provides a simple case of this general phenomenon, and we can state it as follows.
Given a real-analytic vector field $X$ on~$\Tt^d$, one asks whether there exists a real-analytic diffeomorphism that conjugates $X$ to a constant vector field $\omega \in \Rr^d$. 
%
%
%

We can now state precisely our quantitative linearization result.  
Let $f \colon \Tt^d \to \Rr^d$ be real analytic, i.e. $f \in C^\omega(\Tt^d)$, with Fourier coefficients $(f_k)_{k \in \Zz^d}$. 
Given $\rho > 0$, define the norms
$$
|f|_\rho := \sum_{k \in \Zz^d} |f_k| e^{\rho |k|}
\quad \text{and} \quad
|f|'_\rho := \sum_{k \in \Zz^d} |f_k| (1 + |k|) e^{\rho |k|}.
$$
The parameter $\rho$ controls the exponential decay of the Fourier coefficients, corresponding to the width of the analyticity strip around $\Tt^d$.
Recall that a diffeomorphism $h$ of $\Tt^d$ conjugates  $f$ to the vector field $h^*f:=(Df)^{-1}f\circ h$.
 

\begin{theorem}\label{thm: main}
Let $\omega = (\alpha,1) \in \Rr^d$, $\rho > 0$, and $X \in C^\omega(\Tt^d)$. Assume that
\begin{enumerate}
\item $\omega$ belongs to  the rotation set of $X$,
\item $r:=\rho -( a + b B_\omega)>0$, where $a =  1+ 8\log(2^{10}|\omega|)(2 \cdot 5^{d-1} + 1)$ and $b = 6d$,
\item $|X-\omega|'_\rho \le \frac{1}{496 |\omega|}$.
\end{enumerate}
Then, $X$ is analytically conjugate to $\omega$, that is, there exists a real analytic diffeomorphism $H$ of $\Tt^d$ such that $H^*X=\omega$, and
$$
|H - \id|_r \le 22\bigl(2 \cdot 5^{d-1} + 1\bigr)\, |X-\omega|'_\rho.
$$
\end{theorem}


Notice that it is only required that the vector field $X$ has $\omega$ inside its rotation set (see section~\ref{sec:rotation vector} for the definition). However, as the conjugacy to $\omega$ holds, this vector is actually the only element of the rotation set.

We also would like to highlight the fact that we have an explicit definition of the amplitude of the perturbation.
On the other hand, the size of the analyticity parameter $\rho$ depends on the value of the function $B_\omega$, which, in general, is hard to compute.
Large values of $B_\omega$ imply necessarily large values of $\rho$, which imposes strong restrictions on the allowed vector fields around $\omega$; thus making the theorem a local result.
It is however possible that a global result could be obtained in a situation that it is possible to reduce to the  local one.

In another perspective, a quantitative result can be used to show linearization of $C^k$, $C^\infty$ and Gevrey flows using approximations by analytic maps, in the spirit of Moser~\cite{Moser1966II} (see also~\cite{Popov2004,Bounemoura2022,chierchia2003kam}).

The proof of Theorem~\ref{thm: main} improves methods in the context of renormalization and extends results that appeared in e.g.~\cite{jld5,jld10,DiasGaivao2019} and references therein.
The theorem should also be compared with the  approach of R\"ussmann and P\"oschel using KAM theory. In \cite{Poschel2010KAMR}, P\"oschel proved a linearization result in this context using a new variant of KAM theory,  introduced by R\"ussmann, which is based on a slowly converging iteration scheme.  The Brjuno condition  is formulated in terms of a R\"ussmann approximation function $R\colon[1,\infty)\to[1,\infty)$
satisfying
$$
\int_1^\infty \frac{\log R(t)}{t^2}\,dt<\infty,
$$
and the condition on  $\omega$ is given by
$$
|k\cdot\omega|\ge \frac{C}{R(|k|)},
\qquad 
k\in\Zz^d\setminus\{0\},
$$
for some constant $C>0$.
The resulting loss of analyticity is then controlled by the tail integral
$$
\int_\tau^\infty \frac{\log (t R(t))}{t^2}\,dt.
$$

Our result differs in several respects.  First, our method is of a different nature: instead of solving successive homological equations with prescribed small-divisor estimates, we eliminate Fourier modes according to a sequence of resonant cones whose widths are determined dynamically by the orbit of the associated unimodular lattice under the diagonal flow.
As a consequence, no approximation function is prescribed a priori, and the relevant quantity controlling the loss of analyticity is the Brjuno function $B_\omega$, defined intrinsically from this orbit. The Brjuno condition then emerges naturally as the summability condition governing the cumulative loss of analyticity, and the scheme yields explicit quantitative estimates for the conjugacy.
Finally, the analyticity loss depends linearly on $B_\omega$, with constants depending only on the dimension $d$ and $|\omega|$,  providing an explicit quantitative control of the conjugating transformation. This type of estimate can be compared with classical bounds on the convergence radius in the Poincaré--Siegel problem (see for instance \cite{GM10} and references therein).

The problem of determining the optimal arithmetic condition for analytic linearization remains largely open in $d\ge3$. 
Recent developments suggest that weaker than Brjuno conditions may suffice in certain contexts. 
For instance,  Bounemoura discusses the optimal conditions for the linearization of Gevrey vector fields~\cite{Bounemoura2022}, while Argentieri and Corsi recently announced a weaker-than-Brjuno condition for analytic torus diffeomorphisms~\cite{ArgentieriCorsi2024}.

\medskip

The paper is organized as follows. 
Section~\ref{section:Preliminaries} introduces basic notations. 
Section~\ref{sec:diophantine approx} develops the Diophantine approximations arising from the diagonal flow and establishes the equivalence of several Brjuno-type conditions. 
The subsequent section implements the analytic elimination scheme and prove linearization under the Brjuno condition.
In the appendices we include properties of analytic function spaces and of lattices.

%
%
%

\section{Preliminaries}\label{section:Preliminaries}

Let $d\geq2$. We set the notations $\Nn$ for the positive integers and $\Nn_0:=\Nn\cup\{0\}$ for the non-negative ones. 
The $\ell_1$-norm on $\Cc^d$ is denoted by
$$
|v|:=\sum_{i=1}^d|v_i|.
$$ 
The canonical inner product between vectors $u,v\in\Cc^d$ is given by 
$$
u\cdot v:=\sum_iu_iv_i
$$ 
and it satisfies
$
|u\cdot v| \leq |u|\,|v|.
$

Define the norm
$$
\|v\|_*:=\max\{ |\hat v|,|v_d|\}
$$
where  $v=(\hat v,v_d)\in\Cc^d$ with 
$$
\hat v=(v_1,\dots,v_{d-1})\in\Cc^{d-1}.
$$

%

Let $C^r(\Tt^d,\Rr^d)$ with $r\in\Nn_0\cup\{\infty,\omega\}$ denote the set of $C^r$ vector-valued functions on the $d$-torus \( \Tt^d := \Rr^d / 2\pi\Zz^d \). Their lifts to the universal cover give $2\pi\Zz^d$-periodic $C^r$ vector-valued functions.  
Given any $f\in C^r(\Tt^d,\Rr^d)$, its Fourier coefficients are
$$
f_k=\frac{1}{(2\pi)^d}\int_{\Tt^d}f(x)e^{-ik\cdot x}\,dx,\quad k\in\Zz^d,
$$
and write the constant Fourier coefficient of $f$ through the projection 
$\Ee \colon f\mapsto f_0$.
For the sup-norm of $f$ we mean $\|f\|_{C^0}:=\sup_{x\in\Tt^d}|f(x)|$. 

We will use multi-index notation. So given $\alpha=(\alpha_1,\ldots,\alpha_d)\in\Nn_0^d$ we write
$$
\alpha!=\alpha_1!\cdots \alpha_d!\,,\quad |\alpha|=\alpha_1+\cdots+\alpha_d\quad\text{and}\quad \partial^\alpha=\partial^{\alpha_1}_{x_1}\cdots\partial^{\alpha_d}_{x_d}
$$
for the partial derivatives.

\section{Diophantine approximations}
\label{sec:diophantine approx}

	Consider the homogeneous space $\SLZ\backslash \SLR$. This space is identified with the space of $d$-dimensional unimodular lattices, which means that for $M\in \SLR$, we identify the lattice $\Zz^d M=\{k^\top M:\veck\in\Zz^d \}$ with the coset $\SLZ M$.
	Consider the right action of the one-parameter subgroup
	$$
	E^t=\diag(\e^{-t}, \dots, \e^{-t}, \e^{(d-1)t}) \in \SLR
	$$
	that generates the flow
	\begin{equation}\label{def geod flow}
		\Phi^t\colon \ \SLZ\backslash \SLR \to\SLZ\backslash \SLR,
		\qquad
		\Zz^d M\mapsto \Zz^d M E^t.
	\end{equation}
	This flow is known to be ergodic~\cite{Bekka2000}. 
	In the following we will be interested in the properties of one particular orbit: $\Phi^t(\Lambda_\omega)$, $t\geq 0$, where the vector $\omega=(\alpha,1)\in\Rr^d$ is given and
	 \begin{equation}
		M_{\omega}=
		\begin{pmatrix}
			I & \alpha^\top \\
			0 & 1
		\end{pmatrix} \text{ and } \Lambda_{\omega}=\Zz^d M_{\omega}.
	\end{equation}

	The size of the shortest non-zero vector in a lattice $\Zz^dM\in \SLZ\backslash \SLR$ is given by the function  
	\begin{equation}
		\delta\colon  \SLZ\backslash\SLR\to\Rr^+,
		\qquad
		\delta(\Zz^dM)=\inf_{\veck\in\Zz^d\setminus \{0\}} \| k^\top M\|_*.
	\end{equation}
		The forward orbit $\Phi^t(\Lambda_{\omega})$, for $t \geq 0$, encodes all the Diophantine properties of the linear form $\Rr^d\ni x\mapsto \omega\cdot x$ associated with $\omega$ that we require. Key information is captured by the functions
	\[
	W,\Delta\colon[0,+\infty) \to \Rr, \qquad W(t) = \log \delta(\Phi^t(\Lambda_{\omega})), \quad \Delta(t) = t + W(t).
	\]
	These functions can be expressed in terms of the sequence of best approximations to the linear form associated with $\omega$.
	
	\subsection{Best Approximations}
	
	Recall that $\Rr^{d-1}$ is endowed with the $\ell^1$-norm and $\Rr^d$ with the norm $\|v\|_* = \max(|\hat v|, |v_d|)$.
	
	\begin{defn}
		Let $\omega = (\alpha,1) \in \Rr^d$. A vector $p \in \Zz^d$ is called a \emph{best approximation vector} to the linear form $\omega$ if $|\hat p| > 0$ and for all $b \in \Zz^d$, the following holds:
		\[	
		\left\{\begin{array}{ll}
			0 < |\hat b| < |\hat p| \Rightarrow |b \cdot \omega| > |p \cdot \omega|,\\
			|\hat b| = |\hat p| \Rightarrow |b \cdot \omega| \geq |p \cdot \omega|.
		\end{array}
		\right.
		\]
		The  projection $\hat p\in\Zz^{d-1}$ is referred to as \emph{the short best approximation vector} associated with $p$ and $|p \cdot \omega|$ is referred to as the \emph{approximation error}.
	\end{defn}
	Given $\omega =(\alpha,1)\in \mathbb{R}^d$,  for each possible value of the norm $|\hat{p}|$ of a short best approximation vector to $\omega$, we select a best approximation vector of that norm. We then order these vectors increasingly according to the norm of their associated short vectors. This procedure yields a sequence $(p_n)_{n\geq 0}$ such that
	\[
	|\hat p_0| = 1 < |\hat p_1| < \dots < |\hat p_n| < \dots
	\]
	Although the sequence $(p_n)_n$ is not uniquely defined, the sequences $(|\hat{p}_n|)_n$ and $(|p_n \cdot \omega|)_n$ are uniquely determined. These sequences are infinite if and only if the coordinates of $\omega$ are linearly independent over $\mathbb{Q}$.  Such $\omega$'s are known as \textit{rationally independent}.


Best approximations to linear forms naturally appear in the analysis of the function $W$ (see Lemma \ref{lem:W} below) and in the estimations of Birkhoff sums of toral translations via Fourier coefficients. By contrast, simultaneous best approximations do not.  
	However for comparison with the Brjuno-type condition introduced by Golse and Lochak~\cite{GoLo}, we  recall the notion of best simultaneous approximations.
	
	\begin{defn}
		Let $\omega = (\alpha,1) \in \Rr^d$.
		An integer vector $q \in \Zz^d$ is a \emph{best simultaneous approximation vector} to $\omega$ if $q_d > 0$ and for all $b \in \Zz^d$, the following holds:
		\[
		\left\{\begin{array}{ll}
			0 < b_d < q_d \Rightarrow |b_d\alpha - \hat b| > |q_d\alpha - \hat q|, \\
			b_d = q_d \Rightarrow |b_d\alpha - \hat b| \geq |q_d\alpha - \hat q|.
		\end{array}
		\right.
		\]
		The integer $q_d$ is called the \emph{best approximation denominator}, and $|q_d\alpha - \hat q|$ is the corresponding \emph{approximation error}.
	\end{defn}
	
	Given $\omega=(\alpha,1)\in \Rr^d$, for each possible value of the denominator $q_d$ of a best simultaneous approximation vector to $\omega$, we select a best approximation vector with that denominator and order these vectors increasingly according to their denominators. This yields a sequence $(q_n)_{n\geq0}$ satisfying:
	\[
	q_{0,d} = 1 < q_{1,d} < \dots < q_{n,d} < \dots
	\]
	This sequence is finite if and only if $\alpha \in \mathbb Q^{d-1}$. 
	
	Typically, one refers to best approximations to $\alpha$ rather than to $\omega$. However, in the context of  flows on the torus, it is more natural to consider $\omega$.

	\subsection{Local maxima and  minima of the function $W$}
	Given $\omega=(\alpha,1)\in\Rr^d$, let $(p_n)_{n\geq0}$ be the sequence of best approximation vectors to the linear form $\omega$.
	\begin{lemma}\label{lem:W}
		The function $W\colon[0,+\infty)\to\Rr$ is continuous piecewise affine with slopes $-1$ and $d-1$:
		\[
		W(t) =
		\begin{cases}
			-t + \log|\hat p_n|, & t \in [\tau_n, s_n]\\
			(d-1)t + \log|p_n \cdot \omega|, & t \in [s_n, \tau_{n+1}],
		\end{cases}
		\]
		where the switching points $\tau_n$ are the local maxima and the $s_n$ are the local minima,  given by
\begin{eqnarray*}
\tau_0&=&0,\\
\tau_n &=& \frac{1}{d} \log \left( \frac{|\hat p_n|}{|p_{n-1} \cdot \omega|} \right),
\quad
n\geq 1, \\
s_n &=& \frac{1}{d} \log \left( \frac{|\hat p_n|}{|p_n \cdot \omega|} \right),
\quad
n\geq 0,
\end{eqnarray*}
		and satisfy
		\[
		 \tau_0 < s_0 < \tau_1 < \dots < \tau_n < s_n < \tau_{n+1} < \dots
		\]
	\end{lemma}

	\begin{proof}
		Since lattices are discrete, for each $t$, there exists $u \in \Zz^d$ such that $u^\top M_\omega E^t$ is a shortest vector of the lattice $\Zz^d M_\omega E^t$. Let $u$ be a nonzero vector in $ \Zz^d$. If $\hat u=0$ then $u\cdot\omega=u_d$ and  $\|u^{\top }M_{\omega}E^t\|_*=|u_d|e^{(d-1)t}\geq \max(e^{-t}|\hat p_0|,e^{(d-1)t}|p_0\cdot \omega|)$ for all $t\geq 0$. If $\hat u \neq 0$, there exists a best approximation vector $p_n$ to the linear form associated with $\omega$ such that $|\hat p_n| \leq |\hat u|$ and $|p_n \cdot \omega| \leq |u \cdot \omega|$. Thus
		\[
		\|p_n^{\top} M_\omega E^t\|_* \leq \|u^{\top} M_\omega E^t\|_*.
		\]
		It follows that, for all $t\geq 0$,
		\[
		W(t) = \log \min_{n \geq 0} \|p_n^{\top} M_\omega E^t\|_* = \min_{n \geq 0} \max\left(-t + \log|\hat p_n|, (d-1)t + \log|p_n \cdot \omega|\right).
		\]
		Thus, $W$ is a continuous piecewise affine function with two possible slopes: $-1$ and $d-1$. Now, the lemma follows from the facts that for each $n$, the function $W_{p_n}(t) = \max(-t + \log|\hat p_n|, (d-1)t + \log|p_n \cdot \omega|)$ reaches its minimum at $s_n$, and $W_{p_n}(t) \leq W_{p_{n+1}}(t)$ if and only if $t \leq \tau_{n+1}$.
	\end{proof}

\begin{remark}\label{rem:Minkowski}By Lemma~\ref{lem:W}, $\Delta$ is a continuous piecewise affine function with
slopes $0$ and $d$, satisfying $\Delta(0)=0$. Thus $\Delta(t)\leq dt$ for every $t\geq0$. Moreover, by Minkowski's theorem, $W(t)\leq \log(((d-1)!)^{1/d})\leq \log(d-1)$ for every $t\geq0$. Hence, $\Delta(t)\leq t+\log(d-1)$ for every $t\geq0$.
\end{remark}

\subsection{Brjuno's conditions}
\label{sec:Brjuno conditions}
Let $\Delta_n:=\Delta(\tau_n)$. In this work we require the following Brjuno-type condition to hold:
\[
 B_{\omega}:=\sum_{n\geq 0}e^{-\Delta_n}\tau_{n+1}<\infty;
\]
see Theorem  \ref{thm:infrenormalization}. 
In Proposition \ref{prop:Br1} below, we show that this technical condition is equivalent to the more readable integral condition
\[
\hspace{2cm}I_{\omega}:=\int_{0}^{\infty}e^{-\Delta(t)}dt=\int_{0}^{\infty}\frac1{e^t\delta(\Phi^t(\Lambda_{\omega}))}<\infty,
\] 
as well as to the linear forms Brjuno's condition, 
\[
 L_{\omega}:=\sum_{n\geq 0}\frac{|\log |p_n\cdot\omega||}{|\hat p_n|}<\infty.
\]
The  Brjuno's condition for Fourier series is given by 
\[
 O_{\omega}:=\sum_{j\geq 0}2^{-j}|\log\min\{|k\cdot\omega|:k\in\Zz^d\text{ and }0<|\hat k|<2^j\}|<\infty.
\] 	

\begin{remark}   
    In~\cite{Brjuno,Brjuno2}, the ``Brjuno condition'' is defined as $\sum_{j \ge 0} \frac{|\Omega_j|}{2^j} < \infty$, where
    \[
    \Omega_j = \log \inf \left\{ |Q \cdot \omega| \colon Q \in \mathbb{N}_d, \, Q \cdot \omega \neq 0, \, \sum_{i} Q_i \le 2^j \right\}
    \]
    and $\mathbb{N}_d$ is the set of vectors $Q \in \mathbb{Z}^d$ with non-negative coordinates, except for at most one coordinate which may equal $-1$ provided that the sum of the coordinates is non-negative. Brjuno introduced this condition in his study of analytic differential equations near a singular point. He proved that, under this condition, there exists an analytic change of variables that transforms the system into a normal form. Since  Brjuno's work involves power series rather than Fourier series,  the integer vectors $Q$ are restricted to the subset $\mathbb{N}_d$ of $\mathbb{Z}^d$.   
    This  condition with $\mathbb N_d$ instead of $\mathbb Z^d$ is less restrictive than the condition $O_{\omega} < \infty$. Indeed, if $\omega = (\omega_1, \dots, \omega_{d-1}, 1)$  is such that $O_{\omega} = \infty$, then all vectors $\omega' = (\pm \omega_1, \dots, \pm \omega_{d-1}, 1)$ also satisfy $O_{\omega'} = \infty$, whereas, if all the coordinates of $\omega'$ are positive,  the Brjuno condition mentioned above is satisfied.
\end{remark}

Golse and Lochak showed in~\cite{GoLo} that the condition $ O_{\omega}<\infty$ implies yet another one, involving the sequence $(q_n)_n$ of best simultaneous Diophantine approximations to $\omega=(\alpha,1)$,		
	\[
	S_{\omega}:=\sum_{n\geq 0}\frac{\log |q_{n,d}\alpha-\hat q_n|}{q_{n,d}^{1/(d-1)}}<\infty.
	\]
	
\begin{proposition}\label{prop:Br1}
	For any $\omega=(\alpha,1)\in\Rr^d$, we have the equivalence
	\[
	O_{\omega}<\infty\iff L_{\omega}<\infty\iff I_{\omega}<\infty\iff B_{\omega}<\infty. 
	\] 
	Moreover, there exists a vector $\omega=(\alpha,1)\in\Rr^d$ with rationally independent coordinates such that $S_{\omega}<\infty$ while $L_{\omega}=\infty$, i.e. $L_\omega\subsetneq S_\omega$.
\end{proposition}

The first part of the proposition follows from the effective inequalities provided in Theorem~ \ref{prop:Br2} below. The second part will be established  in the next subsection.
  
Let 
$$
g_{d-1}:=2\times 5^{d-1}+1.
$$
According to Lemma \ref{lem:egr} in Appendix C, we  have $|\hat p_{n+g_{d-1}}|\geq 2|\hat p_n|$ for all $n\geq 0$,  i.e. the sequence $(|\hat p_n|)_n$ grows exponentially. 
Consequently,
\begin{equation}\label{sum 1/p_n}
\sum_{n\geq0} e^{-\Delta_n} = \sum_{n\geq0} \frac{1}{|\hat{p}_n|} \leq 2 g_{d-1}.
\end{equation}
This exponential growth  was first proved by Lagarias (\cite{Lag82,lagarias1983}). In our framework, the exponential growth rate is easy to prove, see Appendix ~\ref{appendix C}.

\begin{theorem}\label{prop:Br2}
	For any $\omega=(\alpha,1)\in\Rr^d$, we have the following inequalities
	\begin{enumerate}
		\item $\frac12 O_{\omega}\leq L_{\omega}\leq 2g_{d-1} O_{\omega}$
		\item $\tfrac1d L_{\omega}\leq B_{\omega}\leq aL_{\omega}+b$ where $a=\frac{1}{d-1}$ and $b=2g_{d-1}\frac{\log(d-1)!}{d(d-1)}$,
		\item $ I_{\omega}-\tfrac1d\leq B_{\omega}\leq a'I_{\omega} +b'$ where $a'=ad(4g_{d-1}+2)$  and  $b'=a(2g_{d-1}+1)(-2+\log|p_0\cdot\omega|)+b$.
	\end{enumerate}	
\end{theorem}

\begin{remark} $O_{\omega}<\infty\iff L_{\omega}<\infty$ and $O_{\omega}<\infty\implies S_{\omega}<\infty$ are well known (see ~\cite{DiasGaivao2019} and ~\cite{GoLo}).
	\end{remark}

\begin{remark}
	The function $\delta(\LL)$ is defined with respect  to the norm $\|.\|_*$ on $\Rr^{d}$. However the convergence of the integral $I_{\omega}$ does not depend on a particular choice of a norm on $\Rr^d$. 
	The others conditions, namely $O_{\omega}<\infty$ and $L_{\omega}<\infty$, are also independent of the choice of the norm on $\Rr^{d-1}$. It is fairly obvious in the case of $O_{\omega}<\infty$. As  for $L_{\omega}<\infty$, the norm independence follows from a result on best approximation vectors associated with different norms; see Appendix ~\ref{appendix C}.   
\end{remark}

\begin{remark}\label{rem:lineardep} 

	
If $x\cdot\omega=0$ for some nonzero vector $x\in\Zz^d$, then for all $t$,
		\[
		\delta(\Zz^dM_{\omega}E^t)\leq \|x^TM_{\omega}E^t\|_*= e^{-t}|\hat x|,
		\] 
		and therefore 	$\int_{0}^{\infty}\frac{1}{e^t\delta(\Zz^dM_{\omega}E^t)}dt=\infty$. Moreover,   $p_n\cdot\omega=0$ for some $n$ and there are only finitely many best approximations of the linear form. Accordingly, the quantity $|\log|p_n.\omega||$  is understood to be $+\infty$.
\end{remark}

\begin{remark}
	If $d=2$, we have $\tfrac12\leq |\hat p_{n+1}||p_n\cdot\omega|\leq 1$, hence  
	$$
	\sum_{n\geq 0}\frac{|\log |p_n\cdot\omega||}{|\hat p_n|}<\infty\iff \sum_{n\geq 0}\frac{\log |\hat p_{n+1}|}{|\hat p_n|}<\infty.
	$$ 
	Furthermore, when $d=2$ the best approximations to the linear form coincides with the best approximation to the vector so that the condition $L_{\omega}<\infty$ is equivalent to the classical condition  $\sum_{n\geq 0}\frac{\log q_{n+1}}{q_n}<\infty$.
\end{remark}

\begin{remark}
	When $d\geq 3$, due to the existence of singular linear forms, the quantity $|p_n\cdot\omega||\hat p_{n+1}|^d$ can be arbitrarily closed to zero. So the  condition $L_{\omega}<\infty$ is not equivalent to a condition involving only  the sequence   of  short vectors norm, i.e., $(|\hat p_n|)_n$.
\end{remark}

\begin{remark}\label{remark 311}
	If the linear form $\omega$ has a finite Diophantine exponent -- that is, if there exists $s<\infty$ such that for all nonzero $k\in\Zz^d$, we have $|k\cdot\omega|\geq C/|\hat k|^s$ for some $C>0$ -- then $\omega$ satisfies the linear forms Brjuno condition. This follows from the exponential growth rate of the sequence  $(|\hat p_n|)_n$ (see Lemma \ref{lem:egr}). 
 Then, it follows from  Jarn\'ik's theorem for linear forms with $d-1$ variables, that the set of linear forms associated with vectors $\omega=(\alpha,1)$, that do not satisfy the linear forms Brjuno condition, has   Hausdorff dimension at most $d-2$ (see \cite{BoveyDodson1986}).  Taking remark \ref{rem:lineardep} into account, we see that the set of linear forms that do not satisfy the  equivalent  Brjuno conditions has   Hausdorff dimension exactly equal to $d-2$.
\end{remark}

\begin{proof}[Proof of Theorem~\ref{prop:Br2} (1)]
 For $j\in\Nn$, let denote $\varepsilon_j=\min\{|p\cdot\omega|\colon p\in\Zz^d,0<|\hat p|\leq 2^j\}$ and $\mathcal I_j=\{n\in\Nn\colon 2^{j-1}<|\hat p_n|\leq 2^j\}$. 
 By Lemma \ref{lem:egr}, $\#\mathcal I_j\leq g_{d-1}$. Furthermore, $\frac{|\log|p_n\cdot\omega||}{|\hat p_n|}\leq \frac{|\log\varepsilon_j|}{2^{j-1}}$ for $n\in\mathcal I_j$. Therefore,
 \[
 	L_{\omega}=\sum_{n\geq 0}\frac{|\log|p_n\cdot\omega||}{|\hat p_n|}=\sum_{j\geq 0} \sum_{n\in\mathcal I_j}\frac{|\log|p_n\cdot\omega||}{|\hat p_n|}
 	\leq \sum_{j\geq 0}\#\mathcal I_j\frac{|\log\varepsilon_j|}{2^{j-1}}=2g_{d-1}O_{\omega}.
 	\]
 	Next, for all $j\geq 0$, if $n$ is maximal with $|\hat p_n|\leq 2^j$, then $\varepsilon_j= |p_n\cdot\omega|$. Therefore,
 	\begin{align*}
 		O_{\omega}=\sum_{j\geq 0}\frac{\varepsilon_j}{2^j}\leq \sum_{j\geq 0}\sum_{n\geq 0\colon |\hat p_n|\leq 2^j}\frac{|\log|p_n\cdot\omega||}{2^j}=\sum_{n\geq 0}\sum_{j\geq 0\colon 2^j\geq |\hat p_n|}\frac{|\log|p_n\cdot\omega||}{2^j} \leq 2L_{\omega}.
 	\end{align*}  	
\end{proof}	
\begin{proof}[Proof of Theorem~ \ref{prop:Br2} (2)]
	In Lemma \ref{lem:W} about the function $W$, we have seen that
	 $\tau_0=0$, $\tau_{n+1} = \frac{1}{d} \log \left( \frac{|\hat p_{n+1}|}{|p_{n} \cdot \omega|} \right)$  and  $\Delta(\tau_n)=W(\tau_n)+\tau_n=\log|\hat p_n|$ for all $n\geq 0$. Hence
	\begin{align*}
		B_{\omega}=\sum_{n\geq 0}e^{-\Delta(\tau_n)}\tau_{n+1}&=\frac1d\sum_{n\geq 0}\frac{\log|\hat p_{n+1}|-\log|p_n\cdot\omega|}{|\hat p_n|}.
	\end{align*}
	Hence, $B_{\omega}\geq \tfrac1d L_{\omega}$. 
	Now by Lemma \ref{lem:Dirichlet}, $|\hat p_{n+1}|^{d-1}|p_n\cdot\omega|\leq (d-1)!$, hence
	\begin{align*}
		B_{\omega}\leq \frac{\log(d-1)!}{d(d-1)}\sum_{n\geq 0}\frac1{|\hat p_n|}+\frac1{d-1}L_{\omega}. 
	\end{align*}
	By Lemma \ref{lem:egr}, $|\hat p_{n+g_{d-1}}|\geq 2|\hat p_n|$, hence $\sum_{n\geq 0}\frac1{|\hat p_n|}\leq 2g_{d-1}$ and (2) holds.  
\end{proof}	
\begin{proof}[Proof of Theorem~ \ref{prop:Br2} (3)]
	We need two lemmas. The second lemma enables us to exploit the exponential growth of the sequence of best approximation vectors. We will use this  lemma with $a_n=-\log|p_n\cdot\omega|$, $b_n=|\hat p_n|$ and $n_0=g_{d-1}$.

\begin{lemma}\label{formula:integralDelta}
	\begin{align*}
	\int_{0}^{\infty}e^{-\Delta(t)}dt&=\frac1{d}+\sum_{n\geq 0}e^{-\Delta(\tau_n)}(s_n-\tau_n)\\
	&=\frac1{d}\big(1-\log|p_0\cdot\omega|+\sum_{n\geq 1}\frac1{|\hat p_n|} (-\log|p_n\cdot\omega|+\log |p_{n-1}\cdot\omega|)\big)
	\end{align*}
\end{lemma}
\begin{proof}
	Straightforward computation.
\end{proof}

\begin{lemma}\label{lem:Abel}
	Let $(a_n)_{n\geq 0}$ be a non decreasing sequence of non-negative real numbers and  $(b_n)_{n\geq 0}$ be a non decreasing sequence of positive real numbers. If there exists $n_0\in\Nn$ such that $b_{n+n_0}\geq 2 b_n$ for all $n\geq 0$, then
	\[
	\sum_{n\geq 1}\frac{a_n}{b_n}\leq 2n_0\frac{a_0}{b_0}+(4n_0+2)\sum_{n\geq 1}\frac{a_{n}-a_{n-1}}{b_n}.
	\]
\end{lemma}
\begin{proof}
	For all $n\geq 1$, we have 
\begin{align*}
	\sum_{k=n}^{n+n_0}\frac{a_{k}-a_{k-1}}{b_k}
	&=\sum_{k=n}^{n+n_0-1}\big(\frac{1}{b_k}-\frac{1}{b_{k+1}}\big)a_k-\frac{a_{n-1}}{b_n}+\frac{a_{n+n_0}}{b_{n+n_0}}.
\end{align*}
Since the sequences  $(a_n)_{n}$ and $(b_n)_n$ are non decreasing, it follows that
\begin{align*}
	\sum_{k=n}^{n+n_0}\frac{a_{k-1}-a_k}{b_k}\geq 
	\big(\frac{1}{b_n}-\frac{1}{b_{n+n_0}}\big)a_n-\frac{a_{n}}{b_n}+\frac{a_{n+n_0}}{b_{n+n_0}}.
\end{align*}
Since $a_n\geq 0$ and $b_{n+n_0}\geq 2b_n>0$, 
\begin{align*}
	\sum_{k=n}^{n+n_0}\frac{a_{k-1}-a_k}{b_k}\geq 
	\frac{a_n}{2b_n}-\frac{a_{n}}{b_n}+\frac{a_{n+n_0}}{b_{n+n_0}}.
\end{align*}
Therefore, for all $N\geq 1$,
	\begin{align*}
		(n_0+1)\sum_{k= 1}^{Nn_0+n_0}\frac{a_{k}-a_{k-1}}{b_k}
		&\geq\sum_{n=1}^{Nn_0}\sum_{k=n}^{n+n_0}\frac{a_k-a_{k-1}}{b_k} \geq\sum_{n= 1}^{n_0N} \frac{a_n}{2b_n}-\sum_{n=1}^{n_0}\frac{a_{n}}{b_n}.
	\end{align*}
Now, for each $n\in\{1,\dots,n_0\}$, we have
\begin{align*}
	\sum_{k=1}^{n_0}\frac{a_k-a_{k-1}}{b_k}\geq \sum_{k=1}^n\frac{a_k-a_{k-1}}{b_k}\geq \sum_{k=1}^n\frac{a_k-a_{k-1}}{b_n}=\frac{a_n-a_0}{b_n}.
\end{align*}
Hence,
\[
 \sum_{n=1}^{n_0}\frac{a_{n}}{b_n}\leq a_0\sum_{n=1}^{n_0}\frac{1}{b_n}+n_0\sum_{k=1}^{n_0}\frac{a_k-a_{k-1}}{b_k}. 
\] 
Therefore, 
\begin{align*}
	\sum_{n= 1}^{n_0N} \frac{a_n}{2b_n}&\leq (n_0+1)\sum_{n= 1}^{Nn_0+n_0}\frac{a_{n}-a_{n-1}}{b_n} + a_0\sum_{n=1}^{n_0}\frac{1}{b_n}+n_0\sum_{k=1}^{n_0}\frac{a_k-a_{k-1}}{b_k}\\
	&\leq (2n_0+1)\sum_{n\geq 1}\frac{a_n-a_{n-1}}{b_n}+n_0\frac{a_0}{b_0}
\end{align*}
and the lemma follows.
\end{proof}

{\it End of proof of Theorem~ \ref{prop:Br2} (3).}
By Lemma \ref{lem:Abel}, with 
$$
a_n=-\log|p_n\cdot\omega|,
$$ 
$b_n=|\hat p_n|$ and $n_0=g_{d-1}$,
\[
\sum_{n\geq 1}\frac{-\log|p_n\cdot\omega|}{|\hat p_n|}\leq -2g_{d-1}\log|p_0\cdot\omega|+(4g_{d-1}+2)\sum_{n\geq 1}\frac{-\log|p_n\cdot\omega|+\log|p_{n-1}\cdot\omega|}{|\hat p_n|}. 
\]
By Lemma \ref{formula:integralDelta},
\[
I_{\omega}=\frac1d\big(1-\log|p_0\cdot\omega|+\sum_{n\geq 1}\frac{-\log|p_n\cdot\omega|+\log|p_{n-1}\cdot\omega|}{|\hat p_n|}\big),
\]
hence, 
$
I_{\omega}\leq \frac1d(1+L_{\omega})
$, and 
\begin{align*}
	I_{\omega}&\geq \frac1d\big(1-\log|p_0\cdot\omega|+\frac{1}{4g_{d-1}+2}[2g_{d-1}\log|p_0\cdot\omega|+L_{\omega}+\log|p_0\cdot\omega|]\big)\\
	&=\frac1d\big(1-\frac{1}{2}\log|p_0\cdot\omega|+\frac{1}{4g_{d-1}+2}L_{\omega}\big).
\end{align*} 
Now by (2), $\tfrac1d L_{\omega}\leq B_{\omega}\leq aL_{\omega}+b$, hence
\[
I_{\omega}-\tfrac1d\leq B_{\omega}\leq a'I_{\omega}+b'
\]  
where
\[
a'=ad(4g_{d-1}+2) \text{ and } b'=a(-(4g_{d-1}+2)+(2g_{d-1}+1)\log|p_0.\omega|)+b.
\]
\end{proof}

\subsection{$S_{\omega}<\infty$ does not imply $L_{\omega}<\infty$ }
\label{sec:non BC}

\begin{theorem}\label{prop: Hausdorff dim}
Suppose that $d\geq 3$ and let $E$ be the set of
	$\omega=(\alpha,1)\in\Rr^d$  such that $S_{\omega}<\infty$,  $\,L_{\omega}=\infty$ and the coordinates of $\omega$ are rationally independent. Then the Hausdorff dimension of $E$ is $d-2$.
\end{theorem}


	\begin{proof}
		In the proof, given a positive integer $D$, for a vector 
		$u=(b_1,\dots,b_{D})\in\Rr^{D}$, 
		we shall use the  notations 
		$\hat u=(b_1,\dots,b_{D-1})$, $u_D=b_D$,
		$|u|=\sum_{i=1}^D|b_i|$ and $|\hat u|=\sum_{i=1}^{D-1}|b_i|$.
		Specifically we shall take $D=d-1$ or $d$.
		
		Let $\omega'=(\alpha',1)\in\Rr^{d-1}$ with rationally independent coordinates and $\alpha'$ a badly approximable vector in $\Rr^{d-2}$, i.e., $|a\alpha'-b|\geq ca^{-1/(d-2)}$ for all $a\in \Zz$ with $a>0$ and all $b\in\Zz^{d-2}$ where $c$ is some positive constant. Observe that $\alpha'$ is also badly approximable as a linear form (see \cite{Cassels57}, chapter V)
		   
		Let  us construct by induction  a sequence of vectors $(a_k)_{k\geq 0}$ in $\Zz^{d-1}$ and a sequence of real numbers $(x_k)_{k\geq 0}$  such that for all $k\geq 1$,
		\begin{enumerate}
			\item $|\hat a_k|\geq 2|\hat a_{k-1}|$,
			\item $x_k\in[0,2e^{-|\hat a_{k-1}|^2}]$,
			\item $a_k\cdot\omega'=X_k+e^{-|\hat a_{k}|^2}$ where $X_k=\sum_{i=0}^k x_k$.
		\end{enumerate} 
		
		Take $x_0=0$, $\hat a_0$ any nonzero vector in $\Zz^{d-2}$ and  $a_{0,d-1}=-\lfloor \hat a_0.\alpha' \rfloor$ (any choice for $a_{0,d-1}$ works).

Assume that the sequences have been constructed up to index $k-1$. We now define $a_k$ and $x_k$.

		Since  $\alpha'$ has at least one irrational coordinate, we can find a vector $ a_k\in\Zz^{d-1}$ such that $|\hat a_k|\geq 2|\hat a_{k-1}|$ and 
		\[
		a_k\cdot\omega'\in [X_{k-1}+e^{-|\hat a_{k-1}|^2},X_{k-1}+2e^{-|\hat a_{k-1}|^2}].
		\]
			Taking 
		\[
		x_k=a_k\cdot\omega'-X_{k-1}-e^{-|\hat a_{k}|^2},
		\]
		conditions (2) and (3) hold.
		
		Let $X=\sum_{k\geq 0}x_k$. Clearly $X<\infty$. If we can prove that  $\omega=(\alpha,1)$,  with $\alpha=(\alpha',X)$, belongs to $E$,  then we are done because the Hausdorff dimension of the set of badly approximable $\alpha'\in\Rr^{d-2}$ is $d-2$ (see \cite{Sch}) and the Hausdorff dimension of $E$ is $\leq d-2$ by Remark~\ref{remark 311}.
		
		First observe that since $\alpha'$ is badly approximable in $\Rr^{d-2}$, there is constant $c>0$ such that  
		\[
		|q_{d}\alpha-\hat q|\ge|q_d\alpha'-\hat{\hat q}|\geq c q_{d}^{-1/(d-2)}
		\] 
		for all $q\in\Zz^d$ with $q_d\neq 0$. Thanks to the exponential growth rate of the denominators of the best approximations vectors $(q_n)_n$ to $\omega$,  it implies that 
		$S_{\omega}<\infty$.
		
		 Now define $v_k=(\hat a_k,-1,a_{k,d-1})$. From condition (3), we have
		 \begin{align*}
		 v_k\cdot\omega&=a_k\cdot\omega'-X\\
		 &=a_k\cdot\omega'-X_{k}-\sum_{i>k}x_i\\
		 &=e^{-|\hat a_{k}|^2}-\sum_{i>k}x_i.
		 \end{align*}
		 By condition (1), 
		 \[
		 \sum_{i>k}e^{-|\hat a_{i-1}|^2}\le\sum_{n\ge 0}e^{-4^n|\hat a_k|^2}\leq \frac32 e^{-|\hat a_{k}|^2},
		 \]
		  and since $0\leq x_i\leq 2e^{-|\hat a_{i-1}|^2}$ by condition (2), we obtain
		 \[
		 | v_k\cdot\omega|\leq  3e^{-|\hat a_{k}|^2}.
		 \]
Let $(p_n)_{n\ge 0}$ be the sequence of best approximation vectors to the linear form $\omega$ and let
$n_k=\max\{n\in\Nn\colon |\hat p_n|\leq |\hat v_k|\}$. By definition of  best approximation vectors,
		 we have
		 \[
		 |\hat p_{n_k}|\leq |\hat v_k| \text{ and } |p_{n_k}\cdot\omega|\leq |v_k\cdot\omega|\leq  3e^{-|\hat a_{k}|^2}.
		 \]	 Therefore 
		 \[
		 \frac{|\log |p_{n_k}\cdot\omega||}{|\hat p_{n_k}|}\geq\frac{|\log 3e^{-|\hat a_{k}|^2}|}{|\hat a_k|+1}\rightarrow\infty 
		 \]
		 when $k$ tends to infinity which implies $L_{\omega}=\infty$.
		 
		 It remains to verify  that the coordinates of $\omega$ are rationally independent. Suppose that $z=(u,\ell,m)$ with $u\in\Zz^{d-2}$ and $\ell,m\in\Zz$, is such that 
		 $
		 z\cdot\omega=0.
		 $ 
		 By (3), $X=a_k\cdot\omega'+y_k$ where $y_k=-e^{-|\hat a_k|^2}+\sum_{i>k}x_i$ has an absolute value $\leq 3 e^{-|\hat a_k|^2}$. It follows that for all $k$,
		 \begin{align*}
		 	 z\cdot\omega=u\cdot\alpha'+\ell(a_k\cdot\omega'+y_k)+m=0
		 \end{align*}
		 which implies that
		 \[
		 |(u+\ell\hat a_k)\cdot\alpha'+\ell a_{k,d-1}+m|\leq 3\ell e^{-|\hat a_k|^2}.
		 \]
		 Letting $k \to \infty$ and using the fact that $\alpha'$ is also badly approximable as a linear form, we deduce that $\ell=0$ and then $u=0$ and $m=0$ because the coordinates of $\omega'$ are rationally independent. 
	\end{proof}

\section{Linearization}

In this section, we start by establishing several preparatory results and later prove Theorem~\ref{thm: main}.

\subsection{Banach spaces}

Given $\rho>0$, let $\A_{\rho}(\Tt^d,\Rr)$ denote the Banach algebra of analytic functions $f\in C^\omega(\Tt^d,\Rr)$ having finite norm
$$
|f|_{\rho}:=\sum_{k\in\Zz^d}|f_k|e^{\rho|k|}.
$$ 
Notice that $|fg|_\rho\leq |f|_\rho|g|_\rho$ for all $f,g\in \A_{\rho}(\Tt^d,\Rr)$.
Because the Fourier coefficients of any $f\in C^\omega(\Tt^d,\Rr)$ decay exponentially, we know that 
$$
C^{\omega}(\Tt^d,\Rr)=\cup_{\rho>0}\A_{\rho}(\Tt^d,\Rr).
$$
To control the derivatives of analytic functions we introduce the following norm,
$$
|f|_{\rho}':=\sum_{k\in\Zz^d}|f_k|(1+|k|)e^{\rho|k|},
$$ 
and define $\A_{\rho}'(\Tt^d,\Rr)\subset\A_{\rho}(\Tt^d,\Rr)$ to be the subset of functions that have the above norm finite. 
We also denote by $\A_{\rho}(\Tt^d,\Rr^d)$ the space of vector-valued functions $f=(f_1,\ldots,f_d)$ where each component belongs to $\A_{\rho}(\Tt^d,\Rr)$.  Notice that $\|f\|_{C^0}\leq |f|_\rho$.  We endow this space with the norm
$$
|f|_{\rho}:=|f_1|_{\rho}+\cdots+|f_d|_{\rho}.
$$
Similarly, we define $\A_{\rho}'(\Tt^d,\Rr^d)$ with norm $|\cdot|'_{\rho}$. It is clear that these spaces are Banach spaces. 
In order to simplify the notation, in what follows we will omit the explicit mention of the domain $\Tt^d$ and $\Rr^d$ in the notation of these spaces. That is, we will write $\A_{\rho}$ and $\A_{\rho}'$ instead of $\A_{\rho}(\Tt^d,\Rr^d)$ and $\A_{\rho}'(\Tt^d,\Rr^d)$, respectively. 
Several properties of these Banach spaces are well-known and easy to establish (see Appendix~\ref{banach spaces}).

\subsection{Elimination of far from resonant modes}

Given $\omega\in\Rr^d\setminus\{0\}$ and $\sigma>0$, we define the set of \textit{far from resonant modes} as
$$
K_\sigma^-(\omega) := \{k\in\Zz^d\colon |k\cdot \omega|>\sigma|k|\}.
$$
Since $\omega$ is fixed throughout this subsection, we shall write $K_\sigma^-$ instead of $K_\sigma^-(\omega)$. The complementary set is denoted by $K_\sigma^+$, called the \textit{resonant modes}. Associated to these sets we define the projections $\Kk^-_\sigma$ and $\Kk^+_\sigma$, acting on the Banach spaces previously defined, by restricting the Fourier modes to $K_\sigma^-$ and $K_\sigma^+$, respectively.  Clearly,  $\Kk_\sigma^++\Kk_\sigma^-=\Ii$,  the identity operator,  and $|\Kk_\sigma^\pm|_{\rho}\leq 1$.

Given $\rho,\varepsilon,\nu>0$, denote by $\Vs_\varepsilon$ the closed ball in $\A'_{\rho+\nu}$ centred at $\omega$ with radius $\varepsilon$, i.e.,
\begin{equation}\label{eq:Vepsilon}
\Vs_\varepsilon = \Vs_\varepsilon(\omega,\rho+\nu) := \{X\in \A'_{\rho+\nu}\colon |X-\omega|'_{\rho+\nu}\leq\varepsilon\}.
\end{equation}

The following theorem is proved in Appendix~\ref{sec:proof thm non-resonant modes}.

\begin{theorem}\label{thm unif}
Given $\rho>0$, $\nu>0$ and $\sigma>0$, let
$$
\varepsilon:=\frac{\sigma\delta}{8}\quad\text{and}\quad \delta:=\min\left\{\frac{\sigma}{14\sigma + 48|\omega|},\frac{\nu}2\right\}.
$$

For every $X=\omega+f\in \Vs_\varepsilon$, there exists $u\in \Kk_\sigma^-\A_{\rho}'$ such that:
\begin{enumerate}
\item $Y:=(\id+u)^*X \in \Kk_\sigma^+\A_{\rho}$,
\item $|u|'_{\rho}\leq\frac{8}{\sigma}|f|'_{\rho+\nu}$,
\item $|Y-\omega|_{\rho}\leq 2|f|'_{\rho+\nu}$.
\end{enumerate}
\end{theorem}

\subsection{Norm improvement}

Given $\omega\in\Rr^d\setminus\{0\}$ and $\sigma>0$, 
let
\begin{equation}\label{eq:ell}
\ell_{\sigma}=\ell_{\sigma}(\omega):=\frac{1}{\displaystyle\inf_{k\in K^+_\sigma(\omega)\setminus\{0\}}|k|}.
\end{equation}
Notice that $0<\ell_{\sigma}\leq1$.  The following estimates are immediate.

\begin{lemma}\label{lem:estimate ell} If $\sigma(t)=e^{-dt}$,  then 
$\ell_{\sigma(t)} \leq 2e^{-\Delta(t)}$ for every $t\geq0$.
\end{lemma}
\begin{proof}
For every $t\geq0$ and $k\in K^+_{\sigma(t)}$,  
$$
\|k^\top M_\omega E^t\|_* \leq |k^\top M_\omega E^t| = e^{-t}|\hat{k}|+ e^{(d-1)t}|k\cdot \omega| \leq 2e^{-t} |k|,
$$
since $|k\cdot \omega|\leq e^{-dt}|k|$ and $|\hat{k}|\leq |k|$.  Thus,
$$
\delta(t)=\inf_{k\neq0} \|k^\top M_\omega E^t\|_* \leq  2 e^{-t}\inf_{k\in K^+_{\sigma(t)}}|k| = \frac{2e^{-t}}{\ell_{\sigma(t)}},
$$
from which the estimate follows.
\end{proof}

\begin{lemma}\label{lem:norm improvement}
If $X\in \Kk^+_\sigma\A_{\rho}$ and $\phi\geq0$, then
$$
|X-\Ee X|_{\rho}\leq e^{-\phi}|X-\Ee X|_{\rho+\ell_{\sigma}\phi}.
$$
\end{lemma}
\begin{proof}
We have,
\begin{align*}
|X-\Ee X|_{\rho}&=\sum_{k\in  K^+_\sigma(\omega)\setminus\{0\}} |X_k|e^{\rho|k|}\\
&\leq \exp\left(-\ell_{\sigma}\phi\min_{k\in K^+_\sigma\setminus\{0\}}|k|\right)|X-\Ee X|_{\rho+\ell_{\sigma}\phi}\\
&=e^{-\phi}|X-\Ee X|_{\rho+\ell_{\sigma}\phi}.
\end{align*}
\end{proof}

%
%

\subsection{Rotation set and the constant term}\label{sec:rotation vector}
Given $X\in C^1(\Tt^d,\Rr^d)$, let $\phi_X^t$ denote the flow of $X$ on $\Tt^d$. Denote by $\Phi_X^t$ the lift of $\phi_X^t$ to the universal cover $\Rr^d$ of $\Tt^d$.

The \textit{rotation set} $\rot(X)$ of $X$ is the set of vectors $\omega\in\Rr^d$ for which there are sequences $(x_n)_{n\geq0}$ in $\Rr^d$ and $(t_n)_{n\geq0}$ in $\Rr^+$ with
$t_n\nearrow+\infty$ such that
\begin{equation}\label{eq: rot set}
\omega=\lim_{n\to+\infty}\frac{\Phi^{t_n}_X(x_n)-x_n}{t_n}.
\end{equation}
The rotation set is a compact and connected set~\cite{zbMATH04084609}.

\begin{lemma}\label{invariant rot}
If $h$ is a diffeomorphism of $\Tt^d$ of the form $h=\id+u$ with $u\in C^1(\Tt^d, \Rr^d)$, then
$$
\rot(h^*X)=\rot(X).
$$
\end{lemma}

\begin{proof}
Write lifts of $h$ and $h^{-1}$ as $H=\id+U$ and $H^{-1}=\id+V$, respectively, with $U,V\in C^1(\Rr^d,\Rr^d)$ that are $\Zz^d$-periodic.
Notice that $\Phi^t_{h^*X}=H^{-1}\circ \Phi^t_X\circ H$.
Thus, 
$$
\Phi^t_{h^*X} -\id
= (\Phi^t_X-\id)\circ H
+V\circ \Phi^t_X \circ H
+U.
$$

Suppose that $\omega\in\rot(X)$.
Given $(x_n)_{n\geq0}$ and $(t_n)_{n\geq0}$ so that ~\eqref{eq: rot set} holds, choose a sequence with general term $y_n=H^{-1}(x_n)$.
Then,
\begin{eqnarray*}
\lim_{n\to+\infty}
\frac{\Phi^{t_n}_{h^*X}(y_n)-y_n}{t_n}
&=&
\lim_{n\to+\infty}
\frac{\Phi^{t_n}_{X}(x_n)-x_n}{t_n}
+\frac{V\circ \Phi^t_X \circ H(y_n)
+U(y_n)}{t_n} \\
&=& \omega +0.
\end{eqnarray*}
That is, $\omega\in \rot(h^* X)$ and $\rot(X)\subset \rot(h^* X)$.
Since $h$ is a diffeomorphism, we can similarly show that $\rot(h^* X)\subset  \rot(X)$.
\end{proof}

For $Y\in C^0(\Tt^d,\Rr^d)$, denote $\|Y\|_{\infty,1}=\sum_{i=1}^d\sup_{x\in\Tt^d}|Y_i(x)|$.

	\begin{lemma}
	\label{lem:omega minus constant}
		If $\omega\in\rot(X)$, 
		then 
		$$
		|\Ee X-\omega|\leq  \|X-\Ee X\|_{\infty,1}.
		$$
	\end{lemma}
	\begin{proof}
Given $(x_n)_{n\geq0}$ and $(t_n)_{n\geq0}$ so that ~\eqref{eq: rot set} holds.
So,
\[
\omega+\varepsilon_n=\frac{\Phi^{t_n}_X(x_n)-x_n}{t_n}=\frac{1}{t_n}\int_0^{t_n}X(\Phi^s_X(x_n))ds
\]
where $\varepsilon_n\rightarrow 0$ as $n\to+\infty$. 
Since the function $s\mapsto X(\Phi^s_X(x_n))$ is continuous, for each $i\in\{1,\dots,d\}$ and each $n$, there exists a time $s_{i,n}\in[0,t_n]$ such that $X_i(\Phi^{s_{i,n}}_X(x_n))=\omega_i+(\varepsilon_n)_i$. Therefore,
		\[
		|\Ee X_i-\omega_i|\leq |\Ee X_i-X_i(\Phi^{s_{i,n}}_X(x_n))|+|(\varepsilon_n)_i|\leq \sup_{y\in\Tt^d}|\Ee X_i-X_i(y)|+|(\varepsilon_n)_i|.
		\]
Summing over $i$, we obtain $|\Ee X-\omega|\leq \|X-\Ee X\|_{\infty,1}+|\varepsilon_n|$, and letting $n$ go to infinity, we obtain $|\Ee X-\omega|\leq \|X-\Ee X\|_{\infty,1}$. 
	\end{proof}

\begin{remark}\label{rem:omega minus constant}
Given any $Y\in\mathcal A_{\rho}(\Tt^d,\Rr^d)$ and $\rho> 0$, we have
		\begin{align*}
			\|Y\|_{\infty,1}&=\sum_{i=1}^d\sup_{x\in\Tt^d}|Y_i(x)|=\sum_{i=1}^d\sup_{x\in\Tt^d}|\sum_{k\in\Zz^d}Y_{i,k}e^{ik\cdot x}|\\
			&\leq\sum_{i=1}^d\sum_{k\in\Zz^d}|Y_{i,k}|\leq \sum_{i=1}^d\sum_{k\in\Zz^d}|Y_{i,k}|e^{\rho|k|}
			=|Y|_{\rho}.
		\end{align*}
Therefore,  under the hypothesis of the previous lemma,  if  $X\in\mathcal A_{\rho}(\Tt^d,\Rr^d)$,  then 
$$
|\Ee X-\omega|\leq\|X-\Ee X\|_{\infty,1}\leq |X-\Ee X|_{\rho}.
$$
A simple use of the triangle inequality gives
$$
2^{-1}|X-\omega|_\rho \leq |X-\Ee X|_{\rho}\leq 2|X-\omega|_\rho.
$$
\end{remark}

\subsection{The $n$-th step operator}
Let $\omega\in\Rr^d$ be rationally independent, i.e. $\omega\cdot k\not=0$ for every $k\in\Zz^d\setminus\{0\}$.  We consider three sequences of positive real numbers:

\begin{itemize}
\item $(\rho_n)_{n\geq0}$  represents the \textit{radius of analyticity of the vectors fields},
\item $(\nu_n)_{n\geq0}$  represents the \textit{loss of analyticity due to the elimination of far from resonant modes},
\item $(\sigma_n)_{n\geq0}$ represents the \textit{width of the cone of the resonant modes}.
\end{itemize}

We call $\omega$ together with these three sequences the \textit{initial data}.  

Given an initial data, we define the following sets,
$$
\Us_{n-1}:=\left\{X\in \Vs_{\varepsilon_{n-1}}\colon \omega\in\rot(X)\right\}, \quad n\geq1,
$$
where $\Vs_{\varepsilon_{n-1}} := \Vs_{\varepsilon_{n-1}}(\omega,\rho_{n-1}+\nu_{n-1})$ is defined in \eqref{eq:Vepsilon} with 
$\varepsilon_{n-1}$ as in Theorem~\ref{thm unif}, 
\begin{equation}\label{eq:varepsilon n minus 1}
\varepsilon_{n-1}:=\frac{\sigma_{n-1}}{8}\min\left\{\frac{\sigma_{n-1}}{14\sigma_{n-1} + 48|\omega|},\frac{\nu_{n-1}}2\right\},
\end{equation}
using $\sigma=\sigma_{n-1}$, $\nu=\nu_{n-1}$ and $\rho=\rho_{n-1}$.
Observe that the sets $\Us_{n-1}$ are determined solely by the initial data. 

The $n$-th step operator $\UU_n\colon \Us_{n-1}\to \A$ is the operator $X\in \Vs_{\varepsilon_{n-1}} \mapsto Y$ given by Theorem~\ref{thm unif}.  Notice that $\omega=\UU_n(\omega)$ (see (1) of subsection~\ref{sec:rotation vector}).

Let
\begin{equation}\label{eq:phin}
\ell_{n-1}:=\ell_{\sigma_{n-1}}\quad \text{and}\quad \phi_n:=\log^+\left(8\frac{\nu_n+1}{\nu_n}\frac{\varepsilon_{n-1}}{\varepsilon_n}\right),
\end{equation}
where $\ell_{\sigma_{n-1}}$ is defined in \eqref{eq:ell}. Notice that both $\ell_{n-1}$ and $\phi_n$ are non-negative. The following result is central to the iterative scheme.
\begin{lemma}[$n$-th step estimate]\label{lem:ren}
If
$\rho_{n-1}\geq \rho_n +2\nu_{n}+\ell_{n-1}\phi_n$ and $X_{n-1}\in \Us_{n-1}$, then
$$
|\UU_n(X_{n-1})-\omega|'_{\rho_n+\nu_n}\leq \frac{\varepsilon_n}{\varepsilon_{n-1}}|X_{n-1}-\omega|'_{\rho_{n-1}+\nu_{n-1}}.
$$
In particular, $\UU_n(\Us_{n-1})\subset \Us_{n}$.
\end{lemma}
\begin{proof}
Given $X_{n-1}\in \Us_{n-1}$, by Proposition~\ref{prop:norms}, Theorem~\ref{thm unif}, Lemma~\ref{lem:norm improvement} and Remark~\ref{rem:omega minus constant}, we have

\begin{align*}
|\UU_n(X_{n-1})-\omega|'_{\rho_n+\nu_n}&\leq\frac{\nu_n+1}{\nu_n} |\UU_n(X_{n-1})-\omega|_{\rho_n+2\nu_n}\tag*{(Proposition~\ref{prop:norms})}\\
&\leq 2\frac{\nu_n+1}{\nu_n}|\UU_{n}(X_{n-1})-\Ee(\UU_{n}(X_{n-1}))|_{\rho_n+2\nu_n}\tag*{(Remark~\ref{rem:omega minus constant})}\\
&\leq2\frac{\nu_n+1}{\nu_n}e^{-\phi_n}|\UU_{n}(X_{n-1})-\Ee(\UU_{n}(X_{n-1}))|_{\rho_{n-1}}\tag*{(Lemma~\ref{lem:norm improvement})}\\
&\leq 4\frac{\nu_n+1}{\nu_n}e^{-\phi_n}|\UU_{n}(X_{n-1})-\omega|_{\rho_{n-1}}\tag*{(Remark~\ref{rem:omega minus constant})}\\
&\leq 8\frac{\nu_n+1}{\nu_n}e^{-\phi_n}|X_{n-1}-\omega|'_{\rho_{n-1}+\nu_{n-1}}\tag*{(Theorem~\ref{thm unif})}.
\end{align*}

By our choice of $\phi_n$ (see equation \eqref{eq:phin}), we have
$$
8\frac{\nu_n+1}{\nu_n}e^{-\phi_n}\leq \frac{\varepsilon_{n}}{\varepsilon_{n-1}}. 
$$
This shows the inequality in the statement of the lemma and $X_n:=\UU_n(X_{n-1})\in \Vs_{\varepsilon_{n}}$. Since $\omega\in \rot(X_{n-1})$,  by Lemma~\ref{invariant rot} we also have $\omega\in \rot (X_n)$.  Thus,  $X_n\in \Us_n$.

\end{proof}

\begin{proposition}\label{prop:ren}
Given $N\geq1$,  if
\begin{enumerate}
\item $\rho_{0}> \sum_{n=1}^{N}2\nu_{n}+\ell_{n-1}\phi_n$,  
\item $\rho_n \leq \rho_{n-1}-2\nu_n-\ell_{n-1}\phi_k$,  $n=1,\ldots,N$,
\end{enumerate}
then $\UU_n(\Us_{n-1})\in \Us_n$ for all $n=1,\ldots,N$.
\end{proposition}
\begin{proof}
Immediate from Lemma~\ref{lem:ren} applying it iteratively through $n=1,\ldots,N$.
\end{proof}

\subsection{Choice of initial data}
Let $\vecomega=(\vecalf,1)\in\Rr^d$ be rationally independent.   Associated to $\omega$, we have the infinite sequence $(\tau_n)_{n\geq0}$ of 
local maximizers of $W$ (see Lemma~\ref{lem:W}).  Define,
$$
\sigma_n:= e^{-d\tau_n}\quad \text{and}\quad \nu_n:=\sigma_n.
$$  
Notice that $\sigma_n\leq 1$ and $\sigma_n \searrow 0$.
These sequences constitute part of the initial data. It remains to specify the sequence $(\rho_n)_{n\geq0}$.  

\subsection{Iterative scheme}
Given $\rho_0>0$, suppose that $X_0\in\Vs_{\varepsilon_0}$ is such that $\omega\in \rot(X_0)$, i.e., $X_0\in\Us_0$,  where $\varepsilon_0$ is given in formula \eqref{eq:varepsilon n minus 1} with $\sigma_0=\nu_0=e^{-d\tau_0}=1$.  Recall that $g_{d-1}=2\times 5^{d-1}+1$.

\begin{theorem}\label{thm:infrenormalization}
If
$
B_\omega=\sum_{n=0}^{\infty} e^{-\Delta_n}\tau_{n+1}<\infty
$
and $(\rho_n)_{n\geq0}$ is defined by
$$
\rho_n=\rho_{0}-4\log(2^{10}|\omega|)\sum_{k=1}^ne^{-\Delta_{k-1}}-6d\sum_{k=1}^n e^{-\Delta_{k-1}}\tau_k,\quad n\geq1
$$
with initial $\rho_0$ satisfying $\rho_0 > 8\log(2^{10}|\omega|)g_{d-1}+ 6d B_\omega$,  then  $\inf_n\rho_n>0$ and
   $$X_n:=\UU_n(X_{n-1})\in \Us_n,\quad\forall\,n\geq1.$$
\end{theorem}
\begin{proof}
We start by showing that $\inf_n\rho_n>0$.   By the hypothesis $B_\omega<+\infty$ and the fact that $\sum_{n\geq0}e^{-\Delta_n}\leq 2 g_{d-1}$ (see \eqref{sum 1/p_n}), we get
\begin{align*}
\inf_n\rho_n &\geq \rho_{0}-4\log(2^{10}|\omega|)\sum_{k=1}^\infty e^{-\Delta_{k-1}}-6d\sum_{k=1}^\infty e^{-\Delta_{k-1}}\tau_k\\
& \geq \rho_0-8\log(2^{10}|\omega|)g_{d-1}-6d B_\omega>0.
\end{align*}

Now we prove that $X_n \in  \Us_n,\quad\forall\,n\geq1$. We need an estimating for $\phi_n$: it follows from \eqref{eq:phin} and $\nu_n=\sigma_n\leq 1$ that
$$
\phi_n = \log^+\left(8\frac{\nu_n+1}{\nu_n}\frac{\varepsilon_{n-1}}{\varepsilon_n}\right)\leq \log^+\left(2^4\frac{\varepsilon_{n-1}}{\sigma_n\varepsilon_n}\right).
$$ 
Moreover,
$$
\varepsilon_n = \frac{\sigma_{n}}{8}\min\left\{\frac{\sigma_{n}}{14\sigma_{n} + 48|\omega|},\frac{\nu_{n}}2\right\} =\frac{\sigma_{n}^2}{2^47\sigma_{n} + 2^73|\omega|}\geq \frac{\sigma_n^2}{2^{10}|\omega|}.
$$
On the other hand, $$\varepsilon_{n-1}\leq \frac{\sigma_{n-1}\nu_{n-1}}{16}\leq \frac{1}{2^4}.$$ 
Putting together these estimates we get,
$$
\phi_n\leq \log^+\left(\frac{2^{10}|\omega|}{\sigma_n^3}\right) = \log (2^{10}|\omega|) + 3d\tau_n,
$$
because $\sigma_n=e^{-d\tau_n}$.   This gives the desired estimate for $\phi_n$.  We also need an estimate for $\ell_n$:  by Lemma~\ref{lem:estimate ell},  we have 
$$
\ell_{n}\leq 2e^{-\Delta_n}.
$$
Using these estimates for $\phi_n$ and $\ell_n$ we get
\begin{align*}
2\nu_n+\ell_{n-1}\phi_n &\leq 2e^{-d\tau_n} + 2e^{-\Delta_{n-1}}(\log (2^{10}|\omega|) + 3d\tau_n)\\
&\leq 4\log (2^{10}|\omega|)e^{-\Delta_{n-1}} + 6d e^{-\Delta_{n-1}}\tau_n,
\end{align*}
since $e^{-d\tau_n}\leq e^{-d\tau_{n-1}}\leq e^{-\Delta_{n-1}}$, where we have used the fact that $\Delta(t)\leq dt$ for every $t\geq0$ (Remark~\ref{rem:Minkowski}).

At this point, we are prepared to verify the assumptions of Proposition~\ref{prop:ren}.  Indeed, consider the sequence $(\rho_n)_{n\geq0}$ defined in the statement of this theorem. Then
\begin{align*}
\rho_{n-1} -\rho_n &= 4\log (2^{10}|\omega|)e^{-\Delta_{n-1}} + 6d e^{-\Delta_{n-1}}\tau_n\\
&\geq 2\nu_n+\ell_{n-1}\phi_n.
\end{align*}
This shows that condition (2) of Proposition~\ref{prop:ren} holds:  $\rho_{n-1}\geq \rho_n +2\nu_{n}+\ell_{n-1}\phi_n$ for all $n\geq1$.  Regarding condition (1),  for any $n\geq1$ we have
\begin{align*}
\rho_0& > 8\log(2^{10}|\omega|)g_{d-1}+ 6d B_\omega\\
&\geq 4\log(2^{10}|\omega|)\sum_{k=1}^\infty e^{-\Delta_{k-1}}+6d\sum_{k=1}^\infty e^{-\Delta_{k-1}}\tau_k\\
&\geq \sum_{k=1}^n \left(4\log(2^{10}|\omega|)e^{-\Delta_{k-1}}+6de^{-\Delta_{k-1}}\tau_k\right)\\
&\geq \sum_{k=1}^n2\nu_k+\ell_{k-1}\phi_k.
\end{align*}
Thus,  by Proposition~\ref{prop:ren}, we conclude that $X_n=\UU_n(X_{n-1})\in \Us_n$ for every $n\geq1$.
\end{proof}

\subsection{Proof of Theorem~\ref{thm: main}}
Let $\vecomega=(\vecalf,1)\in\Rr^d$ be rationally independent such that $B_\omega<+\infty$. Given $\rho_0>0$, suppose that $X_0\in\Vs_0$ with $\omega\in \rot(X_0)$, i.e., $X_0\in\Us_0$. Choosing the sequences $(\sigma_n)_n$, $(\nu_n)_n$ and $(\rho_n)_n$ as in the previous section, if $\rho_0 > 8\log(2^{10}|\omega|)g_{d-1}+ 6d B_\omega$, then by Theorem~\ref{thm:infrenormalization}, 
$$
|X_n-\omega|'_{\rho_n+\nu_n}\leq\varepsilon_n,
\qquad
\forall\,n\geq0,
$$  
where
\begin{equation}\label{eq:epsilon n}
\varepsilon_n=\frac{\sigma_{n}^2}{2^47\sigma_{n} + 2^73|\omega|}\leq \frac{\sigma_{n}^2}{384|\omega|}.
\end{equation}

Notice that
\begin{equation}
X_n=H_n^*(X_0),
\end{equation}
where $H_n := G_1\circ\dots\circ G_n$ and $G_n := \id + u_{n-1}$ with $u_{n-1}\in \A'_{\rho_{n-1}}$ coming from Theorem~\ref{thm unif}.  

\begin{lemma}\label{lem: hypothesis for conjugacy}
$$
\sup_{n\geq1} \frac{|G_n-\id|_{\rho_n}}{\rho_{n-1}-\rho_{n}}\leq \frac12 \quad\text{and}\quad \sum_{n=1}^\infty \frac{|G_n-\id|_{\rho_n}}{\rho_{n-1}-\rho_{n}}\leq (d-1)^2g_{d-1}.
$$
\end{lemma}

\begin{proof}
Given any $n\geq1$ we have that
$$
\rho_{n-1} -\rho_n = 4\log (2^{10}|\omega|)e^{-\Delta_{n-1}} + 6d e^{-\Delta_{n-1}}\tau_n\geq e^{-\Delta_{n-1}}.
$$
By Theorem~\ref{thm unif}, 
$$
|u_n|'_{\rho_n}  \leq \min\left\{\frac{\sigma_{n}}{14\sigma_{n} + 48|\omega|},\frac{\nu_{n}}2\right\} \leq \frac{\nu_{n}}{2}= \frac{\sigma_n}{2}
$$
and
$$
|G_n-\id|_{\rho_{n-1}}=|u_{n-1}|_{\rho_{n-1}} \leq |u_{n-1}|'_{\rho_{n-1}}.
$$
Thus, $|G_n-\id|_{\rho_{n}}\leq |G_n-\id|_{\rho_{n-1}}\leq \frac{\sigma_{n-1}}{2}= \frac12 e^{-d\tau_{n-1}}$. Putting these estimates together we get
$$
\frac{|G_n-\id|_{\rho_n}}{\rho_{n-1}-\rho_{n}}\leq \frac12 e^{\Delta_{n-1}-d\tau_{n-1}}\leq \frac12,
$$
since $\Delta_{n-1}=\Delta(\tau_{n-1})\leq d \tau_{n-1}$ according to Remark~\ref{rem:Minkowski}. This shows the first inequality. Now,  we proceed to prove the second inequality. Using again Remark~\ref{rem:Minkowski} and  \eqref{sum 1/p_n} we obtain
\begin{align*}
\sum_{n=1}^\infty \frac{|G_n-\id|_{\rho_n}}{\rho_{n-1}-\rho_{n}}
&\leq \frac12 \sum_{n=1}^\infty e^{\Delta_{n-1}-d\tau_{n-1}}\\
&\leq\frac12 \sum_{n=1}^\infty e^{\Delta_{n-1}-2\tau_{n-1}}\\
&\leq\frac12 \sum_{n=1}^\infty e^{\Delta_{n-1}-2(\Delta_{n-1}-\log(d-1))}\\
&= \frac12(d-1)^2\sum_{n=1}^\infty e^{-\Delta_{n-1}}\\
&\leq (d-1)^2g_{d-1}.
\end{align*}
\end{proof}

\begin{lemma}
$(H_{n}-\id)_{n\geq1}$ is a Cauchy sequence  in $\A_{\rho_{\infty}}$ with $H:=\lim_nH_n$ being  a diffeomorphism such that 
$$
|H-\id|_{\rho_\infty}\leq 22 g_{d-1}|X_{0}-\omega|'_{\rho_{0}+\nu_{0}}.
$$
where $\rho_\infty := \rho_0 - (8\log(2^{10}|\omega|)g_{d-1}+ 6d B_{\omega})$. 
\end{lemma}

\begin{proof}
That $(H_{n}-\id)_{n\geq1}$ is a Cauchy sequence in $\A_{\rho_{\infty}}$ and the limit $H$ is a diffeomorphism follows from Lemma~\ref{lem: hypothesis for conjugacy} and Proposition~\ref{prop: conv of Hn}. To show the estimate in the norm $|H-\id|_{\rho_\infty}$, again by Proposition~\ref{prop: conv of Hn}, we have
$$
|H-\id|_{\rho_\infty}\leq \sum_{n=1}^\infty |G_n-\id|_{\rho_n}.
$$
As in the proof of Lemma~\ref{lem: hypothesis for conjugacy}, we have $|G_n-\id|_{\rho_n}\leq |u_{n-1}|'_{\rho_{n-1}}$. By Theorem~\ref{thm unif}, $|u_{n-1}|'_{\rho_{n-1}}\leq \frac{8}{\sigma_{n-1}}|X_{n-1}-\omega|'_{\rho_{n-1}+\nu_{n-1}}$. Applying $n-1$ times Lemma~\ref{prop:ren}, we get
\begin{align*}
|X_{n-1}-\omega|'_{\rho_{n-1}+\nu_{n-1}}&\leq \frac{\varepsilon_{n-1}}{\varepsilon_{n-2}}\frac{\varepsilon_{n-2}}{\varepsilon_{n-3}}\cdots\frac{\varepsilon_1}{\varepsilon_0}|X_{0}-\omega|'_{\rho_{0}+\nu_{0}}\\
&=\frac{\varepsilon_{n-1}}{\varepsilon_{0}}|X_{0}-\omega|'_{\rho_{0}+\nu_{0}}.
\end{align*}
This gives
$$
|H-\id|_{\rho_\infty}\leq \frac{8|X_{0}-\omega|'_{\rho_{0}+\nu_{0}}}{\varepsilon_0}\sum_{n=1}^\infty  \frac{\varepsilon_{n-1}}{\sigma_{n-1}}.
$$
Using
successively \eqref{eq:epsilon n}, the fact that $\sigma_{n-1}=e^{-d\tau_{n-1}}$ and that $\Delta(t)\leq dt$ by
Lemma~\ref{lem:W}, we get
$$
\sum_{n=1}^\infty \frac{\varepsilon_{n-1}}{\sigma_{n-1}} \leq \sum_{n=1}^\infty \frac{\sigma_{n-1}^2}{384|\omega|\sigma_{n-1}} = \sum_{n=1}^\infty \frac{e^{-d\tau_{n-1}}}{384|\omega|}\leq\frac{1}{384|\omega|} \sum_{n=1}^\infty e^{-\Delta_{n-1}}.
$$
Thus,
\begin{align*}
|H-\id|_{\rho_\infty}&\leq \frac{8|X_{0}-\omega|'_{\rho_{0}+\nu_{0}}}{384\varepsilon_0|\omega|} \sum_{n=1}^\infty e^{-\Delta_{n-1}}\\
&= \frac{7+24|\omega|}{3|\omega|} \sum_{n=1}^\infty e^{-\Delta_{n-1}}|X_{0}-\omega|'_{\rho_{0}+\nu_{0}}\\
&\leq 22 g_{d-1}|X_{0}-\omega|'_{\rho_{0}+\nu_{0}},
\end{align*}
because $\varepsilon_0= \frac{1}{112+384|\omega|}$, $|\omega|\geq1$ and $\sum_{n=1}^\infty e^{-\Delta_{n-1}}\leq 2g_{d-1}$.
\end{proof}

\begin{lemma}
$H^*(X_0)=\omega$.
\end{lemma} 
\begin{proof}
Using triangle inequality,
\begin{align*}
\|H^*(X_0)-\omega\|_{C^0}&=\|H^*(X_0)-H_n^*(X_0)+H_n^*(X_0)-\omega\|_{C^0}\\
&\leq \|H^*(X_0)-H_n^*(X_0)\|_{C^0}+\|H_n^*(X_0)-\omega\|_{C^0}\\
&\leq \|H^*(X_0)-H_n^*(X_0)\|_{C^0}+|H_n^*(X_0)-\omega|_{\rho_n}\\
&\leq \|H^*(X_0)-H_n^*(X_0)\|_{C^0}+\varepsilon_n.
\end{align*}
By Lemma~\ref{lem:continuity of push-forward}, the first term goes to zero as $n\to\infty$.  The second term also goes to zero as $n\to\infty$ (see  \eqref{eq:epsilon n}). This shows that $H^*(X_0)=\omega.$
\end{proof}


\appendix

\section{Properties of Banach spaces}\label{banach spaces}

%
%

To simplify the notation we write $\A_{\rho}$ and $\A_{\rho}'$ in place of $\A_{\rho}(\Tt^d,\Rr^d)$ and $\A_{\rho}'(\Tt^d,\Rr^d)$, respectively.	
In what follows a function $f\in\mathcal A_{\rho}$ is seen both as a function defined on the torus and as a $2\pi$-periodic function in each variable defined on $\Rr^d$. 	
Given $f\in \A_{\rho}'$ and $h\in \A_{\rho}$, the derivative $Df$ exists and  $Df\,h$ denotes the function $x\in\Tt^d\mapsto Df(x)\,h(x)$. Moreover, we have
$$
|Dfh|_{\rho}\leq |f|'_{\rho}|h|_{\rho},
$$ 
meaning that $Df$ is a bounded operator on $\A_{\rho}$ whenever $f\in\A_{\rho}'$. Denote by $|Df|_{\rho}$ its induced norm. The following properties are well-known; however, for the reader's convenience, we include their proofs here.

\begin{proposition}\label{prop:norms}
Given $\rho>0$ and $\nu>0$, let $f\in \A_{\rho+\nu}$ and $u\in \A_{\rho}$. The following hold:
\begin{enumerate}
\item \label{prop A1 a}
$|Df|_\rho=\max_{|\alpha|=1}|\partial^\alpha f|_\rho\leq \frac1\nu|f|_{\rho+\nu}$.
\item $|f|'_{\rho}=|f|_\rho+\sum_{|\alpha|=1}|\partial^\alpha f|_\rho\leq \left(1+\frac{1}{\nu}\right)|f|_{\rho+\nu}$.
\item If $|u|_\rho\leq \nu$, then 
$$
|f\circ(\id+u)|_\rho\leq |f|_{\rho+\nu}.
$$
\item
If $f_t\in\A_\rho$ for all $t\in[0,1]$, where $t\mapsto f_t(x)$ is integrable for all $x\in\Tt^d$, then
$$
\left| \int_0^1 f_t\,dt \right|_\rho \leq 
\int_0^1 \left| f_t \right|_\rho \,dt.
$$
\item If $|u|_\rho\leq \nu$ and $f\in  \A_{\rho+\nu}'$, then 
$$
|Df\circ(\id+u)|_\rho\leq |Df|_{\rho+\nu}.
$$
\item If $|u|_\rho\leq \nu$, $|v|_\rho\leq \nu$ and $f\in  \A_{\rho+\nu}'$,
$$
|f\circ(\id+u)-f\circ(\id+v)|_\rho\leq |Df|_{\rho+\nu}|u-v|_\rho.
$$
\item If $|u|_\rho\leq \frac\nu2$, $|v|_\rho\leq \frac\nu2$ and $f\in  \A_{\rho+\nu}'$, then 
$$
|Df\circ(\id+u)-Df\circ(\id+v)|_\rho\leq\frac{2}{\nu} |f|'_{\rho+\nu}|u-v|_\rho.
$$
\end{enumerate}
\end{proposition}

\begin{proof}
\hfill
\begin{enumerate}
\item Given $h\in\A_\rho$ with $|h|_\rho=1$, we have 
$$
|Df\,h|_\rho = |\sum_{j=1}^d\partial_{x_j}f  h_j|_\rho\leq \max_{|\alpha|=1}|\partial^\alpha f|_\rho \sum_j|h_j|_\rho = \max_{|\alpha|=1}|\partial^\alpha f|_\rho.
$$ 
Thus $|Df|_\rho\leq \max_{|\alpha|=1}|\partial^\alpha f|_\rho$. On the other hand, let $j\in\{1,\ldots, d\}$ such that $|\partial_{x_j}f|_\rho = \max_{|\alpha|=1}|\partial^\alpha f|_\rho$. Then, setting $h=e_j$ (constant vector equal to zero except in the $j$-th component where it is one), we have $|h|_\rho = 1$ and $|Df\,h|=\max_{|\alpha|=1}|\partial^\alpha f|_\rho$.  This shows the claimed equality. Now, regarding the inequality, from the identity between Fourier coefficients $(\partial^\alpha f)_k =i k^\alpha f_k$ where $k^\alpha=k_1^{\alpha_1}\cdots k_d^{\alpha_d}$, we get $|\partial^\alpha f|_\rho = \sum_k|k^\alpha||f_k|e^{\rho|k|}\leq \sum_k|k||f_k|e^{\rho|k|}\leq\frac{1}{\nu}\sum_k|f_k|e^{(\rho+\nu)|k|}$ since $|k|e^{-\nu|k|}\leq \frac{1}{\nu}$ for all $k\in\Zz^d$.
\item Similar to the previous item, taking into account that $\sum_{|\alpha|=1}|\partial^\alpha f|_\rho = \sum_k|k||f_k|e^{\rho|k|}$. 
\item Given $|u|_\rho\leq \nu$, using the inequality 
$$
|e^{ik\cdot u}|_{\rho}\leq \sum_{n\in\mathbb N}|\tfrac{(ik\cdot u)^n}{n!}|_{\rho}\leq\sum_{n\in\mathbb N}\tfrac{|k|^n|u|_{\rho}^n}{n!}=e^{|k||u|_{\rho}},
$$ 
we obtain
\begin{eqnarray*}
|f\circ(\id+u)|_\rho 
&=& 
\left|\sum_kf_ke^{ik\cdot(x+u(x))}\right|_\rho \\
&\leq & 
\sum_k|f_k||e^{ik\cdot x}|_\rho e^{|k||u|_\rho} \\
&=& 
\sum_k|f_k|e^{(\rho + |u|_\rho)|k|}.
\end{eqnarray*}
\item Since the Fourier coefficients satisfy $(\int f_tdt)_k = \int (f_t)_k dt $, we get
$$
\left|\int f_t dt\right|_\rho =\sum_k \left|\int (f_t)_k dt\,\right| e^{\rho|k|}\leq
\sum_k \int \left|(f_t)_k \right|dt\, e^{\rho|k|}
=\int |f_t|_\rho dt.
$$
\item  By (3), $|(\partial^\alpha f)\circ(\id+u)|_\rho \leq |\partial^\alpha f|_{\rho+\nu}$.  By (1), the claimed inequality follows. 
\item From the equality,
$$
 f\circ (\id+u)-f\circ(\id+v)=\int_0^1 Df\circ (\id + su+(1-s)v)(u-v)\,ds
$$
we get, using (4) and (5), that
\begin{align*}
| f\circ (\id+u)-f\circ(\id+v) |_\rho &\leq \int_0^1 |Df\circ (\id + su+(1-s)v)(u-v)|_\rho\,ds\\
&\leq \int_0^1 |Df\circ (\id + su+(1-s)v)|_\rho|u-v|_\rho\,ds\\
&\leq \int_0^1 |Df|_{\rho+\nu}|u-v|_\rho\,ds\\
&=|Df|_{\rho+\nu}|u-v|_\rho
\end{align*}
since $|su+(1-s)v|_\rho\leq\nu$ for every $s\in[0,1]$.
\item Using (6), for any $|\alpha|=1$, we have
$$
| \partial^\alpha f\circ (\id+u)-\partial^\alpha f\circ(\id+v) |_\rho \leq |D\partial^\alpha f|_{\rho+\nu/2}|u-v|_\rho.
$$
By (1), $|D\partial^\alpha f|_{\rho+\nu/2}\leq \frac{2}{\nu}|\partial^\alpha f|_{\rho+\nu}\leq |f|'_{\rho+\nu}$. From this and the definition of the operator norm we immediately deduce the desired inequality.
\end{enumerate}
\end{proof}

In the following, for any matrix-valued function $A\colon \Tt^d\to \GL_d(\Rr)$ we write $\|A\|_{C^0}=\sup_{x\in\Tt^d}|A(x)|$ where $|\cdot|$ denotes the operator norm in $\Rr^d$ with respect to the $\ell^1$-norm.   For $f\in C^1(\Tt^d,\Rr^d)$ we define the usual $C^1$-norm $\|f\|_{C^1}:=\|f\|_{C^0}+\|Df\|_{C^0}$.  

\begin{lemma}\label{lemma:diffeo}
Let $u\in C^{\omega}(\Tt^d,\Rr^d)$ with $\|Du\|_{C^0}<1$. 
Then $\phi=\id+u$ is an analytic diffeomorphism of $\Tt^d$.
\end{lemma}

\begin{proof}
Let $\Phi\colon \Rr^d\to\Rr^d$ defined by $\Phi(x)=x+U(x)$ where $U$ is the lift of $u$.  Then, $\Phi$ is a lift of $\id+u$.  Clearly,  $\Phi$ is analytic.  Moreover,  it is a local diffeomorphism, since $D\Phi=I+DU$ and $|DU(x)|<1$ for every $x\in\Rr^d$.   
It remains to see that $\Phi$ is injective.  Let $\Phi(x)=\Phi(x')$.  Then $x+U(x)=x'+U(x')$,  which gives, 
\begin{align*}
|x-x'|=|U(x')-U(x)|&\leq \int_0^1|DU(x+s(x'-x))\, (x-x')|\,ds \\
&\leq
\int_0^1 \sum_{j=1}^d|\partial_{x_j}U(x+s(x'-x))| \,|x_j-x_j'|\,ds\\
&\leq
\sum_{j=1}^d \| \partial_{x_j}U \|_{C^0} \,|x_j-x_j'| \\
&\leq
\max_{|\alpha|=1}\|\partial^\alpha U\|_{C^0}
|x-x'|\\
&=
\|Du\|_{C^0}|x-x'|.
\end{align*}
So, we must have $x'=x$.   Thus $\Phi$ is injective.  Therefore $\phi$ is an injective local analytic diffeomorphism of the compact manifold $\Tt^d$.  Hence $\phi$ is surjective, and therefore a bijective local analytic diffeomorphism.  Consequently $\phi^{-1}$
 is analytic,  so $\phi$ is an analytic diffeomorphism of $\Tt^d$.
\end{proof}

\begin{lemma}\label{lem:continuity of push-forward}
Let $X \in C^0(\Tt^d,\Rr^d)$,  $f_n\in C^1(\Tt^d,\Rr^d)$ with $n\geq1$, such that $f_n\to f$ in $C^1(\Tt^d,\Rr^d)$ and both $\phi_n := \id+f_n$ and $\phi := \id+f$ are $C^1$ diffeomorphisms of $\mathbb{T}^d$.
Then
\[
\lim_n \|\phi_n^*X- \phi^*X\|_{C^0}=0.
\]
\end{lemma}

\begin{proof}
Write $\phi_f := \id + f$ and $A_f := (I + Df)^{-1}$ so that $\phi_f^*X = A_f  X\circ\phi_f$.
Then,
\begin{align*}
\phi_{f_n}^*X - \phi_f^*X
&= A_{f_n}X\circ \phi_{f_n} - A_fX\circ \phi_f\\
&= A_{f_n}\bigl(X\circ \phi_{f_n} - X\circ \phi_f\bigr)
+ (A_{f_n} - A_f)X\circ \phi_f.
\end{align*}

Since $f_n \to f$ uniformly and $X$ is uniformly continuous (continuous on compact domain), we have
\[
\|X\circ \phi_{f_n} - X\circ \phi_f\|_{C^0} \to 0.
\]
Since $\phi = \mathrm{Id} + f$ is a $C^1$ diffeomorphism on the compact torus $\Tt^d$, the matrix $I + Df$ is uniformly invertible and
\[
\|A_f\|_{C^0}  < +\infty.
\]
Moreover, since $Df_n \to Df$ uniformly, for $n$ sufficiently large we have
\[
\| A_f(Df_n - Df) \|_{C^0} < 1/2.
\]
From
\[
I + Df_n = (I + Df)\big(I + A_f(Df_n - Df)\big),
\]
and a Neumann series argument, we obtain
\[
\|A_{f_n}\|_{C^0} \le 2 \|A_f\|_{C^0}.
\]

Using the identity
\[
A_{f_n} - A_f
= A_{f_n}(Df - Df_n)A_f
\]
we get
\[
\|A_{f_n} - A_f\|_{C^0}
\le 2 \|A_{f}\|_{C^0}^2 \|Df_n - Df\|_{C^0} \to 0.
\]

Finally, 
\[
\|\phi_{f_n}^*X - \phi_f^*X\|_{C^0}
\le 
\|A_{f_n}\|_{C^0} \|X\circ \phi_{f_n} - X\circ \phi_f\|_{C^0}
+ \|A_{f_n} - A_f\|_{C^0} \|X\circ \phi_f\|_{C^0}.
\]
Each factor is bounded, and both terms tend to zero. 
\end{proof}

In the following the assumption $\phi-\id\in\mathcal A_{\rho}$  means that $\phi\colon\Tt^d\rightarrow\Tt^d$ is a map such that there exists $u\in\mathcal A_{\rho}$ with $\phi=\id+u$.

%
%
%
%
%

\begin{lemma}\label{lemma Hn-id vf v1}
If for each $n\geq1$ we have $0<\rho_{n}<\rho_{n-1}$ and $\phi_n-\id\in\A_{\rho_n}$  such that
$$
|\phi_n-\id|_{\rho_n}\leq \rho_{n-1}-\rho_n,
$$ 
then
$$
|\phi_1\circ\cdots\circ \phi_n -\id|_{\rho_n}\leq \sum_{i=1}^n|\phi_i-\id|_{\rho_i}.
$$
\end{lemma}

\begin{proof}
By writing $\varphi_n=\phi_n-\id \in\A_{\rho_n}$, it is simple to check that
$$
\phi_1\circ\cdots\circ \phi_n-\id=\varphi_n+(\phi_1\circ\cdots\circ \phi_{n-1}-\id)\circ(\id+\varphi_n).
$$
Thus, by Proposition~\ref{prop:norms} inequality (3),
$$
|\phi_1\circ\cdots\circ \phi_n-\id|_{\rho_n}\leq |\varphi_n|_{\rho_n}+|\phi_1\circ\cdots\circ \phi_{n-1}-\id|_{\rho_{n-1}}.
$$
The claim follows immediately.
\end{proof}

\begin{lemma}\label{lemma Hn-Hn-1 v1}
If for each $n\geq1$ we have $0<\rho_{n}<\rho_{n-1}$ and $\phi_n-\id\in\A_{\rho_n}$ such that
$$
|\phi_n-\id|_{\rho_n}
\leq \frac{\rho_{n-1}-\rho_n}{2} ,
$$
then 
\begin{equation}\label{estimate Hn-Hn-1 v1}
|\phi_1\circ\cdots\circ \phi_{n}-\phi_1\circ\dots\circ \phi_{n-1}|_{\rho_{n}}
\leq
\left(1+ 
\frac{2}{\rho_{n-1}-\rho_n}
\sum_{i=1}^{n-1}
|\phi_i-\id|_{\rho_i}
\right)
 |\phi_{n}-\id|_{\rho_n}.
\end{equation}
\end{lemma}

\begin{proof}
Write $h_n=\phi_1\circ\cdots\circ \phi_{n}$ and $\varphi_n=\phi_n-\id \in\A_{\rho_n}$ for any $n\geq1$.
It is simple to check that
\begin{equation*}
\begin{split}
h_{n}-h_{n-1}
&=
\varphi_{n} + 
\sum_{i=1}^{n-1}\left(
\varphi_i\circ F_{i,n}
-\varphi_i\circ F_{i,n-1}
\right),
\end{split}
\end{equation*}
where 
$$
F_{m_1,m_2}:=\phi_{m_1+1}\circ\cdots\circ \phi_{m_2},\quad m_1< m_2,
$$
and $F_{m,m}=\id$.
Clearly, $F_{i,n}=F_{i,n-1}\circ \phi_n$.
So,
$$
h_{n}-h_{n-1}=
\varphi_{n} + 
\sum_{i=1}^{n-1}\left(
\varphi_i\circ F_{i,n-1}\circ(\id+\varphi_n)
-\varphi_i\circ F_{i,n-1}
\right).
$$

From Proposition~\ref{prop:norms} inequalities (6) and (1), 
$$
|\varphi_i\circ F_{i,n}-\varphi_i\circ F_{i,{n-1}}|_{\rho_n}
\leq 
\frac{2}{\rho_{n-1}-\rho_n}
|\varphi_i\circ F_{i,n-1}|_{\rho_{n-1}}
|\varphi_{n}|_{\rho_n}.
$$
Since $\varphi_i\circ F_{i,n-1}=\varphi_i\circ F_{i,n-2}\circ(\id+\varphi_{n-1})$,
again by Proposition~\ref{prop:norms} inequality (3),
\begin{equation*}
\begin{split}
|\varphi_i\circ F_{i,n-1}|_{\rho_{n-1}}
&\leq 
|\varphi_i\circ F_{i,n-2}|_{\rho_{n-2}} \\
&\leq
|\varphi_i|_{\rho_{i}}.
\end{split}
\end{equation*}

Finally,
$$
|h_{n}-h_{n-1}|_{\rho_n}
\leq
\left(1+ 
\frac{2}{\rho_{n-1}-\rho_n}
\sum_{i=1}^{n-1}
|\varphi_i|_{\rho_i}
\right)
 |\varphi_{n}|_{\rho_n}.
$$
\end{proof}

\begin{proposition}\label{prop: conv of Hn}
Suppose that for each $n\geq1$ we have $0<\rho_{n}<\rho_{n-1}$ and $\phi_n-\id\in\A_{\rho_n}$ such that
$$
|\phi_n-\id|_{\rho_n}
\leq \frac{\rho_{n-1}-\rho_n}{2}.
$$
If
$$
\sum_{n=1}^\infty \frac{|\phi_n-\id|_{\rho_n}}{\rho_{n-1}-\rho_{n}}<+\infty\quad\text{and}\quad \rho:=\inf_n \rho_n>0,
$$
then  $(\phi_1\circ\cdots\circ \phi_n-\id)_{n\geq1}$ is a Cauchy sequence in $\A_{\rho}$ and
$$
|h-\id|_\rho \leq \sum_{n=1}^{\infty}
|\phi_n-\id|_{\rho_n}<+\infty,
$$
where $h=\lim_n \phi_1\circ\cdots\circ \phi_n$.
Moreover, if every $\phi_n$ is a diffeomorphism, then  $h$ is also a diffeomorphism.
\end{proposition}

\begin{proof}

Write $h_n = \phi_1\circ\cdots\circ \phi_n$. Then, by Lemma~\ref{lemma Hn-Hn-1 v1} 
\begin{align*}
|h_n-h_{n-1}|_{\rho_{n}}
&\leq
\left(1+ 
\frac{2}{\rho_{n-1}-\rho_n}
\sum_{i=1}^{n-1}
|\phi_i-\id|_{\rho_i}
\right)
 |\phi_{n}-\id|_{\rho_n}\\
 & = \left(\rho_{n-1}-\rho_n+ 
2\sum_{i=1}^{n-1}
|\phi_i-\id|_{\rho_i}
\right)
 \frac{|\phi_{n}-\id|_{\rho_n}}{\rho_{n-1}-\rho_n}\\
 &\leq \left(\rho_{n-1}-\rho_n+ \rho_0-\rho_1+\cdots+\rho_{n-2}-\rho_{n-1}
\right)
 \frac{|\phi_{n}-\id|_{\rho_n}}{\rho_{n-1}-\rho_n}\\
 &\leq ( \rho_0 - \rho)
 \frac{|\phi_{n}-\id|_{\rho_n}}{\rho_{n-1}-\rho_n}.
\end{align*}
Now, given any $\varepsilon>0$, choose $N\geq1$ such that 
$$
\sum_{n=N}^\infty \frac{|\phi_{n}-\id|_{\rho_n}}{\rho_{n-1}-\rho_n} <\frac{\varepsilon}{\rho_{0}-\rho}.
$$
Then, for every $ n>m\geq N$, we have
\begin{align*}
|h_n- h_m|_\rho &\leq \sum_{k=m+1}^{n}|h_{k}-h_{k-1}|_\rho\\
&\leq \sum_{k=m+1}^\infty|h_{k}-h_{k-1}|_\rho \\
&\leq \sum_{k=m+1}^\infty|h_{k}-h_{k-1}|_{\rho_{k}} \\
&\leq \left(\rho_{0}-\rho
\right) \sum_{k=N}^\infty \frac{|\phi_{k}-\id|_{\rho_k}}{\rho_{k-1}-\rho_k}\\
& <\varepsilon.
\end{align*}
This shows that $(h_n-\id)_{n\geq1}$ is a Cauchy sequence. Because $\A_\rho$ is a Banach space, $(h_n)_{n\geq1}$ has a limit $h$ with $h-\id\in\A_\rho$.  Regarding the norm estimate, by Lemma~\ref{lemma Hn-id vf v1}
$$
|h-\id|_\rho = \lim_n |\phi_1\circ\cdots\circ \phi_n-\id|_\rho \leq \lim_n \sum_{i=1}^n|\phi_i-\id|_{\rho_i} <+\infty,
$$
because $ \sum_{i=1}^n|\phi_i-\id|_{\rho_i}\leq \frac{\rho_0-\rho_n}{2}\leq \frac{\rho_0-\rho}{2}$ for all $n\geq1$. 
Finally, if all $\phi_n$ are diffeomorphisms, then we choose $N_1\geq0$ such that $\sum_{n=N_1}^\infty |\phi_n-\id|_{\rho_n}<\frac{\rho}{2}$ and apply the first part of this proposition to the sequence $g_n=\phi_{N_1+n}$ to get $h=\phi_1\circ\cdots \circ \phi_{N_1}\circ \tilde{h}$ with $|\tilde{h}-\id|_\rho<\frac{\rho}{2}$. This shows that $\tilde{h}$ is a diffeomorphism  by Lemma~\ref{lemma:diffeo} since
$$
\|D(\tilde{h}-\id)\|_{C^0}
\leq
\max_{|\alpha|=1}|\partial^\alpha (\tilde h-\id)|_{\rho/2}
\leq
\frac{2}{\rho}|\tilde h-\id|_{\rho}<1,
$$
where we have used Proposition~\ref{prop:norms}~\eqref{prop A1 a}.
Hence, $h$ is also a diffeomorphism.
%
%
\end{proof}

%
%

\section{Proof of Theorem~\ref{thm unif}}
\label{sec:proof thm non-resonant modes}

A statement analogous to Theorem~\ref{thm unif} has been proved in several contexts, see for instance \cite{DiasGaivao2019} and references therein. For the sake of completeness,  we include here a proof.  As in the statement of the theorem, define the constants
$$
\varepsilon:=\frac{\sigma\delta}{8}\quad\text{and}\quad \delta:=\min\left\{\frac{\sigma}{14\sigma + 48|\omega|},\frac{\nu}2\right\},
$$
where $\rho,\sigma,\nu>0$.
From the definition of these constants, the following inequalities are easily verified:
\begin{enumerate}
\item \label{app B ineq 1}
$\varepsilon\leq \frac{\sigma}{8}$,
\item \label{app B ineq 2}
$\delta\leq \min\left(\frac{1}{14}, \frac{\nu}{2}\right)$,
\item \label{app B ineq 3}
$6\delta\varepsilon+\varepsilon+3\delta|\omega|\leq \frac{\sigma}{16}$.
\end{enumerate}
We will use these inequalities along the proof and each time refer to them by their numbers. 
Let $X=\omega + f$ where $f\in \A'_{\rho+\nu}$ and $|f|'_{\rho+\nu}\leq \varepsilon$. 
We seek a near-identity coordinate transformation $U=\id +u$ where $u$ belongs to
$$
\Bs_{\delta}=\{u\in \Kk^-_\sigma\A'_{\rho}\colon |u|'_{\rho}\leq\delta\}.
$$
Notice that, since $|Du|_\rho\leq |u|'_\rho\leq \delta<1$,  the linear map $I+ Du$ is invertible and
$$
U^*X=(I+Du)^{-1}(\omega+f\circ (\id + u)).
$$
By Proposition~\ref{prop:norms}, since $|u|_\rho\leq\nu$ (by inequality (2)), we have a well defined operator $\GG\colon\Bs_{\delta}\to \Kk^-_\sigma\A_{\rho}$ given by,
$$
\GG(u):= \Kk^-_\sigma(I+Du)^{-1}(\omega+f\circ (\id + u)).
$$
Notice that, $\GG(0)=\Kk^-_\sigma X\in\Kk^-_\sigma\A'_{\rho+\nu}$.

We want to find $u\in\Bs_{\delta}$ such that $\GG(u)=0$. We solve this problem using a homotopy, i.e. we will look for a smooth family $u_t\in\Bs_{\delta}$, $t\in[0,1]$, satisfying the equation,
$$
\GG(u_t)=(1-t)\GG(0).
$$


Differentiating with respect to $t$ we conclude that $u_t$ has to satisfy the differential equation
$$
D\GG(u_t)\frac{du_t}{dt}=-\GG(0).
$$
In order to solve this differential equation we invert $D\GG(u)$ and obtain a locally Lipschitz vector field $u\mapsto -D\GG(u)^{-1}\GG(0)$. The following lemmas provide the necessary estimates.

\begin{lemma}\label{lem:DGu}
The map $\GG$ is differentiable in $\Bs_\delta$.  Moreover,  if $u\in\Bs_{\delta}$, then the derivative of $\GG$ at $u$ is a linear operator $D\GG(u)\colon\Kk^-_\sigma\A'_{\rho}\to \Kk^-_\sigma\A_{\rho}$ defined by
\begin{equation}\label{eq:DF}
h\mapsto\Kk^-_\sigma(I+Du)^{-1}\left[(Df)\circ U h-Dh(I+Du)^{-1}(\omega+f\circ U)\right].
\end{equation}
\end{lemma}

\begin{proof}

The map $u\in\Bs_{\delta} \mapsto f\circ(\id+u)$ is differentiable with derivative $Df\circ(\id +u)$.  Indeed,  given $\xi\in \Bs_{\delta} $,  let $\Sigma\colon\Tt^d\to\Rr^d$ be defined by
\begin{align*}
\Sigma(x)&:=f(x + u(x) + \xi(x))-f(x+u(x))-Df(x+u(x))\xi(x) \\
&= \int_0^1 (Df(x + u(x) + s\xi(x))-Df(x+u(x)))\xi(x)\,ds.
\end{align*}
Now,  it follows from Proposition~\ref{prop:norms} (items (4) and (7)) and taking into account that $|u|_\rho\leq\nu/2$ by inequality ~\eqref{app B ineq 2} above, that
\begin{align*}
|\Sigma|_\rho&\leq \int_0^1 |Df\circ(\id + u + s\xi)-Df\circ(\id+u)|_\rho|\xi|_\rho\,ds\\
&\leq \int_0^1  \frac{2}{\nu}|f|'_{\rho+\nu} s |\xi|^2_\rho \,ds\\
&\leq  \frac{1}{\nu}|f|'_{\rho+\nu}  |\xi|^2_\rho.
\end{align*}
On the other hand,  because $|Du|_\rho<1$,  the inverse $(I+Du)^{-1}=\sum_{n=0}^\infty (-Du)^n$ depends analytically on $u$.  Together with the previous observation,  this shows that $\GG$ is differentiable in $\Bs_\delta$.  The expression for the derivative follows from standard computations.

\end{proof}

The next lemma is the only place in the proof of Theorem~\ref{thm unif} where far from resonant modes are directly involved.   

\begin{lemma}\label{lem:DF0} $D\GG(0)^{-1}$ is a bounded linear operator from $\Kk^-_\sigma\A_{\rho}$ to $\Kk^-_\sigma\A'_{\rho}$. Moreover 
$$
|D\GG(0)^{-1}|\leq \frac4\sigma.
$$
\end{lemma}

\begin{proof}
By Lemma~\ref{lem:DGu}, $D\GG(0)\colon\Kk^-_\sigma\A'_{\rho}\to \Kk^-_\sigma\A_{\rho}$ is given by
$$
D\GG(0)h=\Kk^-_\sigma(L_f-D_{\omega})h,
$$
where $L_fh=Df\,h-Dh\,f$ and $D_{\omega} h=Dh\,\omega$. 

We wish to invert $D\GG(0)$. To that end, we start by studying the linear operator $D_{\omega}\colon\Kk^-_{\sigma}\A'_{\rho}\to \Kk^-_{\sigma}\A_{\rho}$.  The Fourier coefficients
of $h\in \A$ are given by
$$
(D_\omega h)_k = \sum_j \omega_j (\partial_j h)_k  = i \sum_j \omega_j  k_j h_k = i (k\cdot\omega) h_k.
$$
Using this formula for the Fourier coefficients, we see that $D_\omega$ is invertible with bounded inverse $D_\omega^{-1}\colon \Kk^-_\sigma\A_{\rho}\to \Kk^-_\sigma\A'_{\rho}$. Indeed, given $h\in \Kk^-_\sigma\A_{\rho}$,  we have
$$
|D_\omega^{-1} h|'_{\rho}= \sum_{k\in K^-_\sigma}(1+|k|)\frac{|h_k|}{|k\cdot\omega|}e^{\rho|k|}
< \sum_{k\in K^-_\sigma}(1+|k|)\frac{|h_k|}{\sigma|k|}e^{\rho|k|}
\leq \frac2\sigma|h|_{\rho},
$$
since $\frac{1+|k|}{|k|}\leq 2$ for every $k\in K^-_\sigma$.
Regarding the operator $L_f$, it is easy to see that $\Kk^-_\sigma L_f$ is a bounded linear operator from $\Kk^-_\sigma\A'_{\rho}$  to $\Kk^-_\sigma\A_{\rho}$. Indeed, given $h\in \Kk^-_\sigma\A'_{\rho}$, we have
$$
|L_f h|_{\rho}\leq |Df\,h|_{\rho}+|Dh\,f|_{\rho}\leq |f|'_{\rho}|h|_{\rho}+|h|'_{\rho}|f|_{\rho}\leq 2|f|_{\rho}'|h|'_{\rho}.
$$
This implies that $D_\omega^{-1} \Kk^-_\sigma L_f$ is a bounded linear operator on $\Kk^-_\sigma\A'_{\rho}$ with norm $|D_\omega^{-1} \Kk^-_\sigma L_f|\leq \frac4\sigma|f|_{\rho}'\leq \frac12$ (by inequality (1)). 
Therefore, $(D_\omega^{-1} \Kk^-_\sigma L_f - \Ii)^{-1}$ is a bounded linear operator on $\Kk^-_\sigma\A'_{\rho}$ with norm $\leq 2$. 
Because $D\GG(0)=\Kk^-_\sigma(L_f-D_{\omega}) = D_\omega(D_\omega^{-1} \Kk^-_\sigma L_f - \Ii)$, the linear operator $(D_\omega^{-1} \Kk^-_\sigma L_f - \Ii)^{-1}D_\omega^{-1}\colon \Kk^-_\sigma\A_{\rho}\to \Kk^-_\sigma\A'_{\rho}$ is the desired inverse of $D\GG(0)$. The bound of the inverse is 
$$
|D\GG(0)|\leq  |(D_\omega^{-1}\Kk^-_\sigma L_f - \Ii)^{-1}D_\omega^{-1}|\leq \frac4\sigma.
$$
\end{proof}

\begin{lemma}\label{lem:DFu0}
If $u\in\Bs_{\delta}$, then the linear operator $D\GG(u)-D\GG(0)$ mapping $\Kk_\sigma^-\A'_{\rho}$ to $\Kk_\sigma^-\A_{\rho}$ is bounded and
$$
|D\GG(u)-D\GG(0)|\leq \frac{\sigma}{8}.
$$
\end{lemma}
\begin{proof}
According to \eqref{eq:DF} we can write $\left(D\GG(u)-D\GG(0)\right)h=\Kk_\sigma^-(I+Du)^{-1}g$
where 
$$
g=(Df)\circ U h-Dh(I+Du)^{-1}(\omega+f\circ U)-(I+Du)\left[Dfh-Dh(\omega+f)\right].
$$
Now, adding and subtracting $(I+Du)Df h$ and expanding the last term in the previous expression, we get 
\begin{align*}
g&=\left[(Df)\circ U -Df-DuDf\right]h + (I+Du)Df h\\
&-Dh(I+Du)^{-1}(\omega + f\circ U) - (I+Du)Df h\\
&+Dh(\omega+f)+DuDh(\omega+f)
\end{align*}
which further simplifies to give
\begin{align*}
g&=A_1+A_2 -Dh(I+Du)^{-1}(\omega + f\circ U) +Dh(\omega+f)\\
&= A_1+A_2 -Dh(I+Du)^{-1}(\omega+f\circ U - (I+Du)(\omega+f))\\
&=A_1+A_2+A_3
\end{align*}
where 
\begin{align*}
A_1&=\left(Df\circ(\id+u)-Df-Du\,Df\right)h,\\
A_2&=Du\,Dh(\omega+f),\\
A_3&=-Dh(I+Du)^{-1}\left(f\circ(\id+u)-f-Du(\omega+f)\right).
\end{align*}
It follows from $|u|'_\rho\leq\frac\nu2$ (inequality (2)) and Proposition~\ref{prop:norms} that,
\begin{align*}
|A_1|_{\rho}&\leq \left(1+\frac{2}{\nu}\right)|f|'_{\rho+\nu}|u|'_{\rho}|h|_{\rho}\leq \left(1+\frac{2}{\nu}\right) \varepsilon\delta |h|'_{\rho},\\
|A_2|_{\rho}&\leq(|\omega| + |f|_{\rho})|u|'_{\rho}|h|'_{\rho}\leq (|\omega| + \varepsilon)\delta|h|'_{\rho},\\
|A_3|_{\rho}&\leq  \frac{|h|'_{\rho}}{1-|u|'_{\rho}}\left[|f|'_{\rho+\nu}|u|_{\rho}+|u|'_{\rho}(|\omega|+ |f|_{\rho})\right]\leq \frac{\delta}{1-\delta}\left(|\omega|+ 2\varepsilon\right)|h|'_{\rho}.
\end{align*}
Thus,
\begin{align*}
|D\GG(u)-D\GG(0)|
&\leq \frac{\delta}{1-\delta}\left(\left(1+\frac{2}{\nu}\right) \varepsilon + |\omega| + \varepsilon + \frac{1}{1-\delta}\left(|\omega|+ 2\varepsilon\right)\right)\\
&\leq 2\delta\left(\left(6+\frac{2}{\nu}\right) \varepsilon + 3|\omega|\right)\\
&\leq 12\delta\varepsilon + 2\varepsilon + 6\delta|\omega|_\rho\\
&\leq \frac{\sigma}{8},
\end{align*}
where we have used $\delta\leq \frac{\sigma}{14\sigma + 48|\omega|}\leq \frac12$ (immediate from the definition of $\delta$), $\frac{2\delta}{\nu}\leq 1$ (inequality~\eqref{app B ineq 2}) and inequality~\eqref{app B ineq 3}).
\end{proof}

\begin{lemma}\label{lem:DFu}
If $u\in\Bs_{\delta}$, then $D\GG(u)^{-1}$ is a bounded linear operator from $\Kk_\sigma^-\A_{\rho}$ to $\Kk_\sigma^-\A'_{\rho}$ with
\begin{align*}
|D\GG(u)^{-1}|\leq \frac{8}{\sigma}.
\end{align*}
Moreover,  $\Bs_{\delta} \ni u\mapsto D\GG(u)^{-1}$ is locally Lipschitz.
\end{lemma}

\begin{proof}
By Lemmas~\ref{lem:DF0} and \ref{lem:DFu0},
\begin{align*}
| (D\GG(u)-D\GG(0))D\GG(0)^{-1} |
\leq \frac{\sigma}{8}\cdot\frac{4}{\sigma}= \frac12.
\end{align*}

Hence the operator
\[
I+(D\GG(u)-D\GG(0))D\GG(0)^{-1}
\]
is invertible and its inverse is given by the convergent Neumann series
\[
\left[I+(D\GG(u)-D\GG(0))D\GG(0)^{-1}\right]^{-1}
=
\sum_{n\ge 0} \left[-(D\GG(u)-D\GG(0))D\GG(0)^{-1}\right]^n.
\]

Therefore
\begin{align*}
D\GG(u)^{-1}
&= D\GG(0)^{-1}
\left[I+(D\GG(u)-D\GG(0))D\GG(0)^{-1}\right]^{-1}.
\end{align*}


From the Neumann series and the estimate above,
\[
\left|
\left[I+(D\GG(u)-D\GG(0))D\GG(0)^{-1}\right]^{-1}
\right|
\leq \sum_{n\ge 0} (1/2)^n = 2.
\]
Hence
\[
|D\GG(u)^{-1}|
\leq 2|D\GG(0)^{-1}|
\leq \frac{8}{\sigma}.
\]

%
Let \(u,v \in \Bs_\delta\). 
Notice that
\begin{align*}
D\GG(u)^{-1}-D\GG(v)^{-1}
=
D\GG(u)^{-1}\bigl(D\GG(v)-D\GG(u)\bigr)D\GG(v)^{-1}.
\end{align*}

Taking norms and using the uniform bound above,
\[
|D\GG(u)^{-1}-D\GG(v)^{-1}|
\le
\frac{64}{\sigma^2}
|D\GG(u)-D\GG(v)|.
\]

Since \(u \mapsto D\GG(u)\) is continuous and differentiable on \(\Bs_\delta\),
it is locally Lipschitz; hence there exists \(L>0\) such that
\[
|D\GG(u)-D\GG(v)|\le L|u-v|'_\rho.
\]

Combining the estimates yields
\[
|D\GG(u)^{-1}-D\GG(v)^{-1}|
\le \frac{64L}{\sigma^2}|u-v|'_\rho.
\]

Thus \(u \mapsto D\GG(u)^{-1}\) is locally Lipschitz.

\end{proof}


We now conclude the proof of Theorem~\ref{thm unif}.  Because $u\mapsto D\GG(u)^{-1}$ is locally Lipschitz,  by Cauchy's existence theorem (e.g. \cite[Theorem 10.4.5 and Remark 10.4.6]{Dieudonne1960}) $u_t$ uniquely exists for every $t\in[0,1]$ and satisfies,
$$
u_t=-\int_0^tD\GG(u_s)^{-1}\GG(0)\,ds.
$$
Since $\GG(0)=\Kk_\sigma^- X\in\Kk_\sigma^-\A_{\rho+\nu}$, it follows from Lemma~\ref{lem:DFu} that, 
\begin{equation}\label{norm of u}
|u_t|'_{\rho}\leq t \sup_{u\in\Bs_\delta}|D\GG(u)^{-1}||\GG(0)|_{\rho+\nu}\leq\frac{8}{\sigma}|f|'_{\rho+\nu}\leq \frac{8}{\sigma}\varepsilon=\delta.
\end{equation}
This guarantees that $u_t\in\Bs_\delta$ for every $t\in[0,1]$. So $X\mapsto u:=u_1$ defines an operator from $\Vs_\varepsilon$ to $\Bs_\delta$ such that $\Kk^-_\sigma(\id+u)^*X = 0$.
In addition, 
$$
(\id+u)^*X-\omega=\Kk^+_\sigma f+\Kk^+_\sigma(A_1 +A_2+A_3)
$$
where 
\begin{align*}
A_1&=Df\,u-Duf-DuDf\,u,\\
A_2&=\left(I-Du\right)\left(f\circ(\id+u)-f-Df\,u\right),\\
A_3&=\sum_{n=2}^\infty(-Du)^n\left(\omega+f\circ(\id+u)\right).
\end{align*}
Using $|u|'_\rho \leq \frac{1}{14}$ (immediate from the definition of $\delta$) and Proposition~\ref{prop:norms},  we get,
\begin{align*}
|A_1|_{\rho}&\leq \frac14|f|'_{\rho+\nu},\\
|A_2|_{\rho}&\leq\frac14|f|'_{\rho+\nu},\\
|A_3|_{\rho}&\leq 2(|u|'_\rho)^2(|\omega|+|f|'_{\rho+\nu}).
\end{align*}
From \eqref{norm of u} we get $|u|'_\rho  \leq \frac{8}{\sigma}|f|'_{\rho+\nu}$, which gives
\begin{align*}
2(|u|'_\rho)^2(|\omega|+|f|'_{\rho+\nu})&\leq  \frac{16}{\sigma}|f|'_{\rho+\nu} \delta (|\omega|+\varepsilon)\\
&=\frac{16}{\sigma} \delta (|\omega|+\frac{\sigma\delta}{8})|f|'_{\rho+\nu}\\
&\leq\delta \frac{16|\omega|+2\sigma}{\sigma}|f|'_{\rho+\nu} \\
&\leq \frac 12|f|'_{\rho+\nu},
\end{align*}
since $\delta\leq \frac{\sigma}{14\sigma+48|\omega|}$,  which follows immediately from the definition of $\delta$.
Thus
$$
|(\id+u)^*X-\omega|_{\rho}\leq |f|'_{\rho+\nu}+|A_1 |_{\rho}+|A_2|_{\rho}+|A_3|_{\rho}\leq 2|f|'_{\rho+\nu}.
$$
This concludes the proof of Theorem~\ref{thm unif}.\qed

\section{Dirichlet's theorem and growth rate of best approximation denominators} \label{appendix C}
	The following lemmas are well-known. The first is equivalent to Dirichlet's theorem for linear forms. We give the easy proof of this lemma.  In \cite{lagarias1983}, Lagarias prove the exponential growth rate of the sequence $(|p_n|)_{n}$ of best approximation vectors. The exponential growth rate of the sequence of short vectors $(|\hat p_n|)_n$ can be deduced from  Lagarias'result. A direct proof is easy and gives a better constant, see below. 
	
	In the definition of best approximations we have used the $\ell^1$-norm on $\Rr^{d-1}$. More generally, best approximations  to linear forms can be defined with any norm on $\Rr^{d-1}$. However, the sequence of best approximations depends on the norm.  The third lemma states that  given two norms, between two consecutive best approximations associated with  the first norm, there cannot be too many best approximations associated with the second norm.
	
	\begin{lemma}\label{lem:Dirichlet} Let $\omega=(\alpha,1)$ be in $\Rr^d$ and let $(p_n)_n$ be the sequence of best approximation vectors to the linear form associated with $\omega$ and the norm $\ell^1$. Then for all $n\geq 0$,
		$|\hat p_{n+1}|^{d-1}|p_n\cdot\omega|\leq (d-1)!$.
	\end{lemma}
	\begin{proof}
		By definition of best approximation vectors, for each $n\in\mathbb N$, the intersection of the cylinder $\CC_n=\{x\in\Rr^d:|x_d|<|p_{n}\cdot\omega|\text{ and } |\hat x|<|\hat p_{n+1}|\}$ with the lattice $\Zz M_{\omega}$, is reduced to the zero-vector. Therefore, by Minkowski's convex body theorem, 
		$\operatorname{Vol}(\CC_n)\leq 2^d|\operatorname{det}M_{\omega}|=2^d$. Now, $\operatorname{Vol}(\CC_n)=\frac{2^d}{(d-1)!}|p_{n}\cdot\omega||\hat p_{n+1}|^{d-1}$, hence $|p_{n}\cdot\omega||\hat p_{n+1}|^{d-1}\leq (d-1)!$.  	
			\end{proof}
	\begin{lemma}\label{lem:egr}
		Let $\omega=(\alpha,1)$ be in $\Rr^d$ and let $(p_n)_{n\geq 0}$ be the sequence of best approximation vectors to the linear form associated with $\omega$ and any norm on $\Rr^{d-1}$. Then for all $n\geq 0$, $|\hat p_{n+g_{d-1}}|\geq 2|\hat p_n|$ where $g_{d-1}=2\times 5^{d-1}$.
	\end{lemma}
	\begin{proof}
		Consider a ball $B(0,2r)$ in $\Rr^{d-1}$ and a finite set $E\subset B(0,2r)$ of points whose mutual distances are all $\geq r$. Since $\bigcup_{x\in E}B(x,r/2)\subset B(0,5r/2)$ and since the balls $B(x,r/2)$, $x\in E$, do not overlap,  
		$$
		\operatorname{Vol}(B(0,5r/2))\geq \#E\times \operatorname{Vol}(B(0,r/2)),$$ 
hence $\#E\leq 5^{d-1}$. 
		
		We want to show that  $|\hat p_{n+2\times 5^{d-1}}|\geq 2|\hat p_n|$. Let us argue by contradiction and suppose that $|\hat p_{n+2\times 5^{d-1}}|< 2|\hat p_n|$. We use the above argument with $r=|\hat p_{n}|$.  Define the sets $E_+$ and $E_-$ of best approximation vectors $p_k$ with $ n\leq k\leq n+2\times 5^{d-1}$, according to the sign of the scalar product $p_k\cdot\omega$. One of these two sets has at least $5^{d-1}+1$ elements, say $E_+$.  Since $\#E_+> 5^{d-1}$, there exists $p_i, p_j\in E_+$ with $i<j$ such that $|\hat p_i-\hat p_j|< |\hat p_n|$.   	By definition of $E_+$ one has $0\leq p_j\cdot\omega<p_i\cdot\omega\leq|p_n\cdot\omega|$, hence $|(p_j- p_i)\cdot\omega|\leq |p_n\cdot\omega|$ contradicting the definition of $p_n$.
	\end{proof}
	
	\begin{lemma}
		Let $|.|_i$, $i=1,2$, be two norms on $\Rr^{d-1}$. Let $\omega=(\alpha,1)\in\Rr^d$. Denote $(p^i_n)_n$ the sequence of best approximation vectors to the linear form associated with $\omega$ and the norm $|.|_i$. Then there exists a constant $N\in\Nn$ depending only on the two norms such that for all $m,n\geq 0$, $|\hat p^1_m|_1\leq |\hat p^2_n|_1$ implies $|\hat p^2_{n+N}|_1\geq |\hat p^1_{m+1}|_1$.    
	\end{lemma}	
	\begin{proof}
		There exists a constant $C>0$ such that $\frac1C |.|_1\leq |.|_2\leq C|.|_1$. Let $k$ be an integer such that $C^2 < 2^k$ and let $N=kg_{d-1}$ where $g_{d-1}=2\times 5^{d-1}$.   Suppose on the contrary that 
		\[
		|\hat p^1_m|_1\leq |\hat p^2_n|_1 \text{ and } |\hat p^2_{n+N}|_1< |\hat p^1_{m+1}|_1.
		\]
		By definition of best approximation vectors associated with $|.|_1$ we have $|p^2_{n+N}\cdot\omega|\geq |p^1_m\cdot\omega|$. Hence, by definition of best approximation vectors associated with $|.|_2$, $|\hat p^2_{n+N}|_2\leq |\hat p^1_m|_2$. Now, 
		\[
		|\hat p^2_{n}|_1\geq \tfrac1C |\hat p^2_{n}|_2\geq \tfrac1C |\hat p^1_{m}|_1\geq \tfrac1{C^2}|\hat p^1_{m}|_2.
		\]
		By the previous lemma and by definition of $N$,
		\[
		|\hat p^2_{n+N}|_2\geq 2^k|\hat p^2_{n}|_2\geq \tfrac{2^k}{C^2}|\hat p^1_{m}|_2>|\hat p^1_{m}|_2
		\]
		a contradiction.
	\end{proof}
	
	\begin{lemma}
		The convergence of the series 
		$$
		L_{\omega}= \sum_{n\geq0} \frac{|\log|p_n\cdot\omega||}{|\hat p_n|}
		$$ does not depend on the norm on $\Rr^{d-1}$ chosen for the definition of best approximation vectors.
	\end{lemma}
	\begin{proof}
		Let $|.|_i$, $i=1,2$, be two norms on $\Rr^{d-1}$. Let $\omega=(\alpha,1)\in\Rr^d$. Denote $(p^i_n)_n$ the sequence of best approximation vectors to the linear form associated with $\omega$ and the norm $|.|_i$. For each $m\in\Nn$, let denote $J_m=\{n\in\Nn:|\hat p^1_{m}|_1\leq |\hat p^2_{n}|_1<  |\hat p^1_{m+1}|_1\}$. By the previous lemma, there exists $N\in\Nn$ such that for all $m$, $\# J_m\leq N$. As in the previous lemma, there is a constant $C>0$ such that $\frac1C |.|_1\leq |.|_2\leq C|.|_1$. By definition of best approximations, for each $n\in J_m$, $|p^2_n\cdot\omega|\geq|p^1_m\cdot\omega|$, hence
		\[
		\sum_{n\in J_m}\frac{-\log|p^2_n\cdot\omega|}{|\hat p^2_n|_2}\leq \sum_{n\in J_m}C\frac{-\log|p^1_m\cdot\omega|}{|\hat p^2_n|_1}\leq NC\frac{-\log|p^1_m\cdot\omega|}{|\hat p^1_m|_1}.  
		\]  
		Consequently, $\sum_{m\in \Nn}\frac{-\log|p^1_m\cdot\omega|}{|\hat p^2_m|_1}<\infty$ implies $\sum_{n\in N}\frac{-\log|p^2_n\cdot\omega|}{|\hat p^2_n|_2}<\infty$.
	\end{proof}

\section*{Acknowledgements}

João Lopes Dias and José Pedro Gaivão are supported by national funds through FCT – Fundação para a Ciência e a Tecnologia, I.P., in the framework of the unit UID/06522/2025.  Antoine Marnat is supported by the French Agence Nationale de la Recherche (ANR), project ANR-25-CE40-1961-01.  Nikolay Moshchevitin's research is supported by Austrian Science Fund (FWF), Forschungsprojekt PAT1961524.

\bibliographystyle{plain} 
\bibliography{rfrncs} 
	
\end{document}